%% file: root.tex
\documentclass[journal,twoside,web]{ieeecolor}

\usepackage{etoolbox}
\makeatletter
\@ifundefined{color@begingroup}%
{\let\color@begingroup\relax
	\let\color@endgroup\relax}{}%
\def\fix@ieeecolor@hbox#1{%
	\hbox{\color@begingroup#1\color@endgroup}}
\patchcmd\@makecaption{\hbox}{\fix@ieeecolor@hbox}{}{\FAILED}
\patchcmd\@makecaption{\hbox}{\fix@ieeecolor@hbox}{}{\FAILED}

\usepackage{generic}
\usepackage{cite}
\input{subtex/pkgsmacros.tex}

\begin{document}

\title{Fully Automated Verification of Linear Systems Using Inner- and Outer-Approximations of Reachable Sets}

% memberships after author name
%\IEEEmembership{Fellow, IEEE}
%\IEEEmembership{Member, IEEE}

\author{Mark Wetzlinger, Niklas Kochdumper, Stanley Bak, and Matthias Althoff
\thanks{
This paragraph of the first footnote will contain the date on which you submitted your paper for review.
This work was supported by the European Research Council (ERC) project justITSELF under grant agreement No 817629, by the German Research Foundation (DFG) project ConVeY under grant number GRK 2428, and by the Air Force Office of Scientific Research and the Office of Naval Research under award number FA9550-19-1-0288, FA9550-21-1-0121, FA9550-22-1-0450 and N00014-22-1-2156.}
\thanks{Mark Wetzlinger and Matthias Althoff are with the Department of Computer Science, Technical University of Munich, 85748 Garching, Germany (e-mail: \{m.wetzlinger, althoff\}@tum.de).
Niklas Kochdumper and Stanley Bak are with the Department of Computer Science, Stony Brook University, Stony Brook, NY 11794, USA (e-mail: \{niklas.kochdumper, stanley.bak\}@stonybrook.edu.}
}

\maketitle

% --------------- ABSTRACT / KEYWORDS ---------------

\input{subtex/abstract}

\begin{IEEEkeywords}
Formal verification, reachability analysis, linear systems, set-based computing.
% see: http://www.ieee.org/organizations/pubs/ani\_prod/keywrd98.txt
\end{IEEEkeywords}
\copyrightnotice

%\rmk{Send to Proofreading Service after last review round!}

% --------------- INTRODUCTION ---------------

\input{subtex/introduction.tex}

% --------------- PRELIMINARIES / PROBLEM FORMULATION ---------------

\input{subtex/pre.tex}

% --------------- AUTOMATED PARAMETER TUNING ---------------

\input{subtex/tuning.tex}

% --------------- INNER-APPROX. / AUTOMATED VERIFICATION ---------------

\input{subtex/post.tex}

% --------------- NUMERICAL EXAMPLES ---------------

\input{subtex/numex.tex}

% --------------- DISCUSSION / CONCLUSION ---------------

\input{subtex/discussion.tex}

\input{subtex/conclusion.tex}

% --------------- APPENDIX  ---------------

\appendices
\input{subtex/addprops.tex}

\input{subtex/lemmas.tex}
%% Appendices, if needed, appear before the acknowledgment!
%
%\section*{Acknowledgment}
%
%\rmk{In most cases, sponsor and financial support acknowledgments are placed in the 
%unnumbered footnote on the first page, not here.}
% ----

% --------------- REFERENCES ---------------

\bibliographystyle{IEEEtran}
\bibliography{subtex/root}

% --------------- BIOGRAPHIES ---------------

\input{subtex/biographies}

\end{document}

%% file: subtex/pkgsmacros.tex
\usepackage{amsmath,amssymb,amsfonts}
\usepackage{algorithm}
\usepackage{algpseudocode}
\usepackage{graphicx}
\usepackage{textcomp}
\usepackage{url}
\usepackage{siunitx} % proper formatting of SI units
\usepackage[hidelinks]{hyperref} % highlighting of referenced equations, sections, etc.
\usepackage{aligned-overset} % for vertical alignment using overset
\usepackage{setspace} % spacing in algorithms
\usepackage{booktabs} % tables

\usepackage{pifont}% http://ctan.org/pkg/pifont
\newcommand{\cmark}{\ding{51}}%
\newcommand{\xmark}{\ding{55}}%

\let\labelindent\relax
\usepackage{enumitem}

\usepackage{amsthm}

\allowdisplaybreaks

\DeclareFontFamily{U}{mathx}{\hyphenchar\font45}
\DeclareFontShape{U}{mathx}{m}{n}{
      <5> <6> <7> <8> <9> <10>
      <10.95> <12> <14.4> <17.28> <20.74> <24.88>
      mathx10
      }{}
\DeclareSymbolFont{mathx}{U}{mathx}{m}{n}
\DeclareFontSubstitution{U}{mathx}{m}{n}
\DeclareMathAccent{\widecheck}{0}{mathx}{"71}

\usepackage{tikz,pgfplots}
\pgfplotsset{compat=1.5}
\usetikzlibrary{calc,arrows,automata}
\usetikzlibrary{shapes,positioning}
\usetikzlibrary{patterns,decorations.pathreplacing} % shading, braces
\usepgfplotslibrary{groupplots}
\usetikzlibrary{pgfplots.groupplots}
\usetikzlibrary{matrix}
\tikzset{
	preEdge/.style={black, <-, shorten <=1pt, >=stealth', semithick, rounded corners=2pt},
	postEdge/.style={black, ->, shorten >=1pt, >=stealth', semithick, rounded corners=2pt},
	axes/.style={black, ->, >=stealth'}
} 

\definecolor{good_blue}{RGB}{0, 92, 171}
\definecolor{good_red}{RGB}{227, 27, 35}
\definecolor{good_gray}{RGB}{230, 241, 238}
\definecolor{good_yellow}{RGB}{255, 195, 37}
\definecolor{good_green}{RGB}{0, 192, 0}

\newcommand\copyrighttext{%
	\footnotesize \copyright 2023 IEEE. Personal use of this material is permitted. Permission from IEEE must be obtained for all other uses, in any current or future media, including reprinting/republishing this material for advertising or promotional purposes, creating new collective works, for resale or redistribution to servers or lists, or reuse of any copyrighted component of this work in other works.}
\newcommand\copyrightnotice{%
	\begin{tikzpicture}[remember picture,overlay]
		\node[anchor=south,yshift=6pt] at (current page.south) {\fbox{\parbox{\dimexpr\textwidth-\fboxsep-\fboxrule\relax}{\copyrighttext}}};
	\end{tikzpicture}%
}

\newcommand{\norm}[1]{\big\lVert#1\big\rVert}
\newcommand{\smallnorm}[1]{\lVert#1\rVert}
\DeclareMathOperator*{\argmax}{arg\,max}
\DeclareMathOperator*{\argmin}{arg\,min}
\DeclareMathOperator*{\diag}{diag}
\DeclareMathOperator*{\rad}{rad}

\newcommand{\bigO}[1]{\mathcal{O}\!\left(#1\right)}

\DeclareMathOperator*{\auxlinComb}{comb}
\DeclareMathOperator*{\auxred}{reduce}
\DeclareMathOperator*{\auxbox}{box}
\DeclareMathOperator*{\auxerr}{err}
\newcommand{\linCombOp}[1]{\auxlinComb\hspace{-1.5pt}\big(#1\big)}	% linear combination
\newcommand{\redOp}[1]{\auxred\hspace{-1.5pt}\big(#1\big)}			% reduction
\newcommand{\boxOp}[1]{\auxbox\!\left(#1\right)}					% box over-approximation
\newcommand{\errOp}[1]{\auxerr\!\left(#1\right)}					% error measurement
\newcommand{\dH}[1]{d_H\big(#1\big)}

\newcommand{\R}[1]{\mathbb{R}^{#1}}
\newcommand{\N}[1]{\mathbb{N}}
\newcommand{\intmat}[1]{\boldsymbol{\mathcal{#1}}}
\newcommand{\vecones}{\mathbf{1}}
\newcommand{\matzeros}{\mathbf{0}}

\newcommand{\B}[1]{\mathcal{B}_{#1}}				% random ball
\renewcommand{\S}[1]{\mathcal{S}_{#1}}				% arbitrary set (overwritten: legal paragraph sign)
\newcommand{\Sover}{\widehat{\mathcal{S}}}			% over-approx of arbitrary set
\renewcommand{\P}[1]{\mathcal{P}_{#1}}				% random polytope (overwritten: computer paragraph sign)
\newcommand{\CZ}[1]{\mathcal{CZ}_{#1}}				% random constrained zonotope
\newcommand{\Z}[1]{\mathcal{Z}_{#1}}				% random zonotope
\newcommand{\zono}[1]{\langle #1 \rangle_{Z}}		% zonotope shorthand
\newcommand{\poly}[1]{\langle #1 \rangle_{H}}		% polytope shorthand
\newcommand{\polyV}[1]{\langle #1 \rangle_{V}}		% polytope shorthand (V-Rep)
\newcommand{\conZono}[1]{\langle #1 \rangle_{CZ}}	% constrained zonotope shorthand

\newcommand{\gens}{\gamma}								% number of zonotope generators
\newcommand{\cons}{h}									% number of constrained zonotope constraints
\newcommand{\polyCons}{a}								% number of constraints of a polytope
\newcommand{\polyVert}{s}								% number of polytope vertices
\newcommand{\tayl}[1]{\eta_{#1}}						% Taylor order of expmat (per step if necessary)
\newcommand{\zonorder}[1]{\rho_{#1}}					% zonotope order
\newcommand{\overzonorder}{\rho_d}						% zonotope order (reduced)
\newcommand{\zonorderU}{\rho_\mathcal{U}}				% zonotope order of input set
\newcommand{\tFinal}{t_{\text{end}}}					% time horizon
\newcommand{\laststep}{k'}								% letter for last step
\newcommand{\laststepzero}{k_0'}						% letter for last step (eredadm)
\newcommand{\nrSafeSets}{r}								% number of safe sets in verification
\newcommand{\nrUnsafeSets}{w}							% number of unsafe sets in verification
\newcommand{\distance}{\nu}								% distance for the verification algorithm
\newcommand{\Gkeep}{G_{\text{keep}}}					% generators kept in reduction
\newcommand{\Gred}{G_{\text{red}}}						% generators reduced in reduction

\newcommand{\initset}{\mathcal{X}^0}		% initial set
\newcommand{\inputset}{\mathcal{U}}			% input set
\newcommand{\inputsetzero}{\mathcal{U}_0}	% input set centered at zero
\newcommand{\uTrans}{\tilde{u}}				% input vector
\newcommand{\C}{\mathcal{C}}				% sum of enlargement terms to account for curvature of trajectories

\newcommand{\safeSet}[1]{\mathcal{G}_{#1}}				% safe set
\newcommand{\unsafeSet}[1]{\mathcal{F}_{#1}}			% unsafe set

\newcommand{\Rtp}[1]{\mathcal{R}(#1)}
\newcommand{\overRtp}[1]{\widehat{\mathcal{R}}(#1)}
\newcommand{\underRtp}[1]{\widecheck{\mathcal{R}}(#1)}
\newcommand{\Rti}[1]{\mathcal{R}(#1)}
\newcommand{\overRti}[1]{\widehat{\mathcal{R}}(#1)}
\newcommand{\underRti}[1]{\widecheck{\mathcal{R}}(#1)}
\newcommand{\Ytp}[1]{\mathcal{Y}(#1)}
\newcommand{\overYtp}[1]{\widehat{\mathcal{Y}}(#1)}

\newcommand{\HPtp}[1]{\mathcal{H}(#1)}
\newcommand{\HPti}[1]{\mathcal{H}(#1)}
\newcommand{\overHPtp}[1]{\widehat{\mathcal{H}}(#1)}
\newcommand{\overHPti}[1]{\widehat{\mathcal{H}}(#1)}
\newcommand{\innerHti}[1]{\widecheck{\mathcal{H}}(#1)}

\newcommand{\Pu}[1]{\mathcal{P}^u(#1)}

\newcommand{\PU}[1]{\mathcal{P}^\inputset{}(#1)}
\newcommand{\overPU}[1]{\widehat{\mathcal{P}}^\inputset{}(#1)}
\newcommand{\overPUzero}[1]{\widehat{\mathcal{P}}_0^\inputset{}(#1)}
\newcommand{\overPUrest}[1]{\widehat{\mathcal{P}}_{\infty}^\inputset{}(#1)}
\newcommand{\underPU}[1]{\widecheck{\mathcal{P}}^\inputset{}(#1)}
\newcommand{\E}[1]{\intmat{E}(\Delta t_{#1},\tayl{#1})}
\newcommand{\Ezero}[1]{\intmat{E}(0,\tayl{#1})}
\newcommand{\Fx}[1]{\intmat{F}(\Delta t_{#1},\tayl{#1})}
\newcommand{\Fu}[1]{\intmat{G}(\Delta t_{#1},\tayl{#1})}
\newcommand{\Fxzero}[1]{\intmat{F}(0,\tayl{#1})}
\newcommand{\Fuzero}[1]{\intmat{G}(0,\tayl{#1})}
\newcommand{\T}[1]{\intmat{T}^{(#1)}}

\newcommand{\emax}{\varepsilon_\text{max}}
\newcommand{\enonaccadm}[1]{\overline{\varepsilon}^{\text{n},\tau}_{#1}}
\newcommand{\eaccadm}[1]{\overline{\varepsilon}^{\text{a}}(#1)}
\newcommand{\eaccadmstep}[1]{\overline{\varepsilon}^{\text{a},\tau}_{#1}}
\newcommand{\eredadm}[1]{\overline{\varepsilon}^{\text{r}}(#1)}
\newcommand{\eredadmstep}[1]{\overline{\varepsilon}^{\text{r},\tau}_{#1}}
\newcommand{\etotalx}[1]{\varepsilon^{x}_{#1}}
\newcommand{\etotaly}[1]{\varepsilon^{y}_{#1}}
\newcommand{\enonacc}[1]{\Delta \varepsilon^{\text{n}}_{#1}(\Delta t_{#1},\tayl{#1})}
\newcommand{\eacc}[1]{\varepsilon^{\text{a}}_{#1}}
\newcommand{\eaccstep}[1]{\Delta \varepsilon^{\text{a}}_{#1}(\Delta t_{#1},\tayl{#1})}
\newcommand{\eaccU}[1]{\varepsilon^{\mathcal{U}}_{#1}}
\newcommand{\eaccUstep}[1]{\Delta \varepsilon^{\mathcal{U}}_{#1}(\Delta t_{#1},\tayl{#1})}
\newcommand{\eUstep}[1]{\Delta \varepsilon^{\mathcal{U},\tau}_{#1}(\Delta t_{#1},\tayl{#1})}

\newcommand{\ered}[1]{\varepsilon^{\text{r}}_{#1}}
\newcommand{\eredstep}[1]{\Delta \varepsilon^{\text{r}}_{#1}(\zonorder{k})}
\newcommand{\eredapx}[1]{\tilde{\varepsilon}^{\text{r}}_{#1}}

\newcommand{\ehom}[1]{\Delta \varepsilon^{\text{h}}_{#1}(\Delta t_{#1},\tayl{#1})}

\newtheoremstyle{style}% name of the style to be used
  {\topsep}% measure of space to leave above the theorem. E.g.: 3pt
  {\topsep}% measure of space to leave below the theorem. E.g.: 3pt
  {\itshape}% name of font to use in the body of the theorem
  {0pt}% measure of space to indent
  {\bfseries}% name of head font
  {:}% punctuation between head and body
  { }% space after theorem head; " " = normal interword space
  {\thmname{#1}\thmnumber{ #2}\thmnote{ (#3)}}

\theoremstyle{style}
\newtheorem{definition}{Definition}
\newtheorem{proposition}{Proposition}
\newtheorem{theorem}{Theorem}
\newtheorem{lemma}{Lemma}
\renewcommand{\algref}[1]{Alg.~\ref{#1}}
\newcommand{\figref}[1]{Fig.~\ref{#1}}
\newcommand{\tabref}[1]{Tab.~\ref{#1}}
\newcommand{\secref}[1]{Sec.~\ref{#1}}

\newcommand{\lineref}[1]{Line~\ref{#1}}
\newcommand{\linesref}[2]{Lines~\ref{#1}-\ref{#2}}

\newcommand{\defref}[1]{Def.~\ref{#1}}
\newcommand{\propref}[1]{Prop.~\ref{#1}}
\newcommand{\thmref}[1]{Theorem~\ref{#1}}
\newcommand{\lmmref}[1]{Lemma~\ref{#1}}
\newcommand{\lmmsref}[2]{Lemmas~\ref{#1}-\ref{#2}}
\makeatletter
\let\OldStatex\Statex
\renewcommand{\Statex}[1][3]{%
  \setlength\@tempdima{\algorithmicindent}%
  \OldStatex\hskip\dimexpr#1\@tempdima\relax}
\makeatother

%% file: subtex/abstract.tex
% abstract must be between 150--250 words ! (otherwise we have to change it accordingly)
\begin{abstract}
Reachability analysis is a formal method to guarantee safety of dynamical systems under the influence of uncertainties.
A substantial bottleneck of all reachability algorithms is the necessity to adequately tune specific algorithm parameters, such as the time step size, which requires expert knowledge.
In this work, we solve this issue with a fully automated reachability algorithm that tunes all algorithm parameters internally such that the reachable set enclosure respects a user-defined approximation error bound in terms of the Hausdorff distance to the exact reachable set.
Moreover, this bound can be used to extract an inner-approximation of the reachable set from the outer-approximation using the Minkowski difference.
Finally, we propose a novel verification algorithm that automatically refines the accuracy of the outer-approximation and inner-approximation until specifications given by time-varying safe and unsafe sets can be verified or falsified.
The numerical evaluation demonstrates that our verification algorithm successfully verifies or falsifies benchmarks from different domains without requiring manual tuning.
% 173 words
\end{abstract}

%% file: subtex/introduction.tex
\section{Introduction}
\label{sec:introduction}

\IEEEPARstart{D}{eploying} cyber-physical systems in safety-critical environments requires formal verification techniques to ensure correctness with respect to the desired functionality, as failures can lead to severe economic or ecological consequences and loss of human life.
One of the main techniques to provide safety guarantees is reachability analysis, which predicts all possible future system behaviors under uncertainty in the initial state and input.
Reachability analysis has already been successfully applied in a wide variety of applications, such as analog/mixed-signal circuits \cite{Frehse2006}, power systems \cite{Pico2014}, robotics \cite{Lengagne2011}, system biology \cite{Kaynama2012}, aerospace applications \cite{Hobbs2018}, and autonomous driving \cite{Vaskov2019}.
The most common verification tasks for these applications are reach-avoid problems, where one aims to prove that the system reaches a goal set while avoiding unsafe sets.
A substantial bottleneck of all verification algorithms is the manual tuning of certain algorithm parameters, which requires expert knowledge.
To overcome this limitation, we present the first fully automated verification algorithm for linear time-invariant systems.

\subsection{State of the Art}
\label{ssec:stateoftheart}

The exact reachable set can only be computed in rare special cases \cite{Gan2018}.
Therefore, one usually computes tight outer-approximations or inner-approximations of the reachable set instead, which can be used to either prove or disprove safety, respectively.
While there exist many different approaches, e.g., stochastic techniques \cite{Vinod2017} or data-driven/learning methods \cite{Devonport2021-CDC,Thorpe2021}, the following review focuses on model-based reachability analysis of linear systems.

Most work is concerned with computing outer-approximations, where the most prominent approaches for linear systems are based on set propagation \cite{Girard2005,Althoff2010a,Frehse2011}.
These methods evaluate the analytical solution for linear systems in a set-based manner and iteratively propagate the resulting reachable sets forward in time.
One can propagate the sets from the previous step or the initial set.
The latter method avoids the so-called wrapping effect \cite{LeGuernic2010}, i.e., the amplification of outer-approximation errors over subsequent steps, but deals less efficiently with time-varying inputs.
An alternative to set propagation is to compute the reachable set using widened trajectories from simulation runs \cite{Donze2007,Dang2008}.
Moreover, special techniques have been developed recently to facilitate the analysis of high-dimensional systems.
One strategy is to decompose the system into several decoupled/weakly-coupled blocks to reduce the computational effort \cite{Kaynama2011-IJC,Bogomolov2018}, while another group computes the reachable set in a lower-dimensional Krylov subspace that captures the dominant dynamical behavior \cite{Althoff2019,Bak2019}.
Inner-approximation algorithms have been researched for systems with piecewise-constant inputs based on set propagation \cite{Girard2006}, piecewise-affine systems based on linear matrix inequalities \cite{Hamadeh2008}, and time-varying linear systems based on ellipsoidal inner-approximations to parametric integrals \cite{Kurzhanskiy2000}.
Other approaches compute inner-approximations by extraction from outer-approximations \cite{Kochdumper2020d} or for projected dimensions \cite{Goubault2019}---despite being designed for nonlinear dynamics, these works can still compete with the aforementioned specialized approaches for linear systems due to their recency.

Concerning the set representations for reachability analysis, early approaches used polyhedra or template polyhedra \cite{Asarin2000}, which are limited to low-dimensional systems due to the exponential increase in the representation size;
ellipsoids \cite{Kurzhanskiy2000} were also used but result in a conservative approximation as they are not closed under Minkowski sum.
More recent approaches use support functions \cite{LeGuernic2010}, zonotopes \cite{Girard2005}, or their combination \cite{Althoff2016c}, for which all relevant set operations can be computed accurately and efficiently.
A related representation is star sets \cite{Duggirala2016}, where constraints are imposed on a linear combination of base vectors.

Some approaches explicitly address the requirement of tightening the computed outer-approximation for successful verification.
Many of them are based on counterexample-guided abstraction refinement (CEGAR), where either the model \cite{Stursberg2004,Frehse2006} or the set representation \cite{Bogomolov2016,Bogomolov2017} is refined.
A more recent approach \cite{Schupp2018} utilizes the relation between all algorithm parameters and the tightness of the reachable sets, proposing an individual parameter refinement in fixed discrete steps to yield tighter results.
Another method \cite{Johnson2016} refines the tightness of the reachable set enclosure as much as the real-time constraints allow for re-computation in order to choose between a verified but conservative controller and an unverified counterpart with better performance.

Common reachability tools for linear systems are
\textit{CORA} \cite{Althoff2015a},
\textit{Flow*} \cite{Chen2013},
\textit{HyDRA} \cite{Schupp2017},
\textit{HyLAA} \cite{Bak2017},
\textit{JuliaReach} \cite{Bogomolov2019a},
\textit{SpaceEx} \cite{Frehse2011},
and \textit{XSpeed} \cite{Ray2015b}.
These tools still require manual tuning of algorithm parameters to obtain tight approximations.
Since this requires expert knowledge about the underlying algorithms, the usage of reachability analysis is currently mainly limited to academia.
Another issue is that the unknown distance between the computed outer-approximation and the exact reachable set, which may result in so-called spurious counterexamples.
Recently, first steps towards automated parameter tuning have been taken:
A rather brute-force method \cite{Bak2016} proposes to recompute the reachable set from scratch, where the parameter values are refined using fixed scaling factors after each run.
However, recomputation is a computationally demanding procedure and does not exploit information about the specific system dynamics.
Another work \cite{Prabhakar2011} tunes the time step size by approximating the flow below a user-defined error bound but is limited to affine systems.
The most sophisticated approach \cite{Frehse2011,Frehse2013} tunes the time step size by iterative refinement to eventually satisfy a user-defined error bound between the exact reachable set and the computed outer-approximation.
Since this error bound is limited to manually selected directions, the obtained information may differ significantly depending on their choice.
In conclusion, there does not yet exist a fully automated parameter tuning algorithm for linear systems that satisfies an error bound in terms of the Haussdorf distance to the exact reachable set.

\subsection{Overview}
\label{ssec:overview}

This work is structured as follows:
After introducing some preliminaries in \secref{sec:preliminaries}, we provide a mathematical formulation of the problem statement in \secref{sec:problemformulation}.
The main body (Secs.~\ref{sec:tuning}-\ref{sec:verification}) builds upon our previous results \cite{Wetzlinger2020}, where we tuned the algorithm parameters based on non-rigorous approximation error bounds using a naive tuning strategy.
Neither inner-approximations nor automated verification/falsification were considered in \cite{Wetzlinger2020}.
In detail, this article contributes the following novelties:
\begin{itemize} %[leftmargin=1.25cm]
	\item %Router tunable
	We derive a rigorous approximation error for the reachable set, based on which we provide an \emph{automated reachability algorithm} (\algref{alg:automated}) that adaptively tunes all algorithm parameters so that any desired error bound in terms of the Hausdorff distance between the exact reachable set and the computed outer-approximation is respected at all times (see \secref{sec:tuning}).
	\item %Rinner tunable
	We show how to efficiently extract an \emph{inner-approximation} from the previously computed outer-approximation (see \secref{sec:innerApprox}).
	\item %verify
	We introduce an \emph{automated verifier} (\algref{alg:verify}), which iteratively refines the accuracy of the outer- and inner-approximations
	until the specifications given by time-varying safe/unsafe sets can be either proven or disproven (see \secref{sec:verification}).
\end{itemize}
\noindent \figref{fig:overview} depicts the relation of individual contributions in Secs.~\ref{sec:tuning}-\ref{sec:verification}.
Finally, we demonstrate the performance of the proposed algorithms on a variety of numerical examples in \secref{sec:numex} and discuss directions for future work in \secref{sec:discussion}.

\begin{figure}[t]
	\setstretch{0.8}
	\setlength{\abovecaptionskip}{0pt}
	\input{./figures/overview.tikz}
	\caption{Overview of theoretical contributions in Secs.~\ref{sec:tuning}-\ref{sec:verification}.}
	\label{fig:overview}
\end{figure}

% papers...
%
%refs/inner-approx/lin:
%zonotopes LTI systems with piecewise constant inputs \cite{Girard2006},
%piecewise-affine systems using linear matrix inequalities \cite{Hamadeh2008},
%time-varying linear systems and ellipsoids \cite{Kurzhanski2000}
%
%refs/inner-approx/nonlin (keep text short):
%box as part of inner-approx \cite{Goubault2014},
%Hamilton-Jacobi \cite{Mitchell2005},
%evolution function \cite{Li2018}
%semi-definite programs \cite{Xue2018,Xue2019}
%time-inverted dynamics \cite{Xue2016,Xue2017},
%Picard iteration \cite{Chen2014},
%only projections \cite{Goubault2017,Goubault2019}

%% file: figures/overview.tikz
\def\xleft{1.25}
\def\xmid{4}
\def\xright{7}
\def\ybot{0.6}
\def\ymid{1.8}
\def\yup{3}
\begin{tikzpicture}[scale = 1,
	roundbox/.style={draw, rounded rectangle, rounded rectangle arc length=90, minimum width=2.5cm, minimum height=0.6cm, align=center, fill=white},
	algbox/.style={draw, minimum width=1.7cm, minimum height=0.6cm, align=center, fill=white},
	contribution/.style={fill=good_gray,draw=black,thin,rounded corners}]

% interpolation + enlargement
%\filldraw[draw=black,dashed,fill=good_gray] (1,0) -- (0.8,1.1) -- (6.3,2.1) -- (6.7,-0.1) -- (1.2,-1.1) -- cycle;
%\node[align=center] at (6.2,-1) {$\tilde{x}^h(t) \oplus \Fx{k} x^*(t_k)$ \\
%	\textcolor{good_blue}{$\tilde{x}^u(t) \oplus \Fu{k} \uTrans{}$}};

% grid for help
%\draw[gray,thin,step=0.5] (0,0) grid (9,4);
% thirds at: 1.25  -  4.25  -  7.25; 0.5 - 1.5 - 2.5

% parts of contributions
\filldraw[contribution] (0,0) -- (0,1.25) -- (2.5,1.25) -- (2.5,2.35)
	-- (0,2.35) -- (0,3.65) -- (5.3,3.65) -- (5.3,0) -- cycle;
\draw (2.625,3.35) circle[radius=0.22cm] node {\small IV};
\filldraw[contribution] (5.4,3.65) -- (8.9,3.65) -- (8.9,2.45) -- (5.4,2.45) -- cycle;
\draw (8.6,2.8) circle[radius=0.22cm] node {\small V};
\filldraw[contribution] (5.4,2.35) -- (8.9,2.35) -- (8.9,0) -- (5.4,0) -- cycle;
\draw (8.6,2) circle[radius=0.22cm] node {\small VI};

% first column
\node[roundbox] (props1to4) at (\xleft,\yup) {\small Prop.~\ref{prop:eps_affine}-\ref{prop:zonorder}, \ref{prop:ecomb}-\ref{prop:errUoneStep} \\ \scriptsize (individual errors)};
\node[algbox] (algstandard) at (\xleft,\ymid) {\small \algref{alg:standard} \\ \scriptsize (base algorithm)};
\node[roundbox] (lemmas1to4) at (\xleft,\ybot) {\small Lemmata~\ref{lmm:E_lim}-\ref{lmm:eUstep_lim} \\ \scriptsize (error behavior)};

% second column
\node[roundbox] (propepsY) at (\xmid,\yup) {\small \propref{prop:eps_Y} \\ \scriptsize (output equation)};
\node[algbox] (algautomated) at (\xmid,\ymid) {\small \textbf{\algref{alg:automated}} \\ \scriptsize (outer-approximation)};
\node[roundbox] (thmautomated) at (\xmid,\ybot) {\small \thmref{thm:automated} \\ \scriptsize (convergence)};

% third column
\node[roundbox] (propminkDiff) at (\xright,\yup) {\small \propref{prop:minkDiff} \\ \scriptsize (inner-approximation)};
\node[algbox] (algverify) at (\xright,\ymid) {\small \textbf{\algref{alg:verify}} \\ \scriptsize (verification)};
\node[roundbox] (props7to8) at (\xright,\ybot) {\small Prop.~\ref{prop:containment}-\ref{prop:isIntersecting} \\ \scriptsize (containment/intersection)};

% arrows from/to props/lemmas/thm
\draw[-stealth'] (props1to4.350) -- (algautomated.north west);
\draw[-stealth'] (lemmas1to4) -- (thmautomated);
\draw[-stealth'] (algautomated) -- (thmautomated);
\draw[-stealth'] (algautomated) -- (propepsY);
\draw[-stealth'] (algautomated.north east) -- (propminkDiff.190);
\draw[-stealth'] (propminkDiff) -- (algverify);
\draw[-stealth'] (props7to8) -- (algverify);

% arrows between algs
\draw[-stealth'] (algstandard.east) -- (algautomated.west);
\draw[-stealth'] (algautomated.east) -- (algverify.west);

\end{tikzpicture}

%% file: subtex/pre.tex
\section{Preliminaries}
\label{sec:preliminaries}

% GENERAL NOTATION

\subsection{Notation}
\label{ssec:notation}

Scalars and vectors are denoted by lowercase letters, matrices are denoted by uppercase letters.
Given a vector $v \in \R{n}$, $v_{(i)}$ represents the $i$-th entry and $\norm{v}_p$ its $p$-norm.
Similarly, for a matrix $M \in \R{m \times n}$, $M_{(i,\cdot)}$ refers to the $i$-th row and $M_{(\cdot,j)}$ to the $j$-th column.
The identity matrix of dimension~$n$ is denoted by $I_n$, the concatenation of two matrices $M_1, M_2$ by $[M_1~M_2]$, and we use $\matzeros{}$ and $\vecones{}$ to represent vectors and matrices of proper dimension containing only zeros or ones, respectively.
The operation $\diag(v)$ returns a square matrix with the vector~$v$ on its diagonal.
Exact sets are denoted by standard calligraphic letters $\S{}$, outer-approximations by $\widehat{\mathcal{S}}$, and inner-approximations by $\widecheck{\mathcal{S}}$.
The empty set is represented by $\emptyset$.
Moreover, we write $v$ for the set $\{v\}$ consisting only of the point $v$.
We refer to the radius of the smallest hypersphere centered at the origin and enclosing a set $\S{}$ by $\rad(\S{})$.
The operation $\boxOp{\S{}}$ denotes the tightest axis-aligned interval outer-approximation of $\S{}$.
Interval matrices are denoted by bold calligraphic letters: $\intmat{M} = [\underline{M},\overline{M}] = \{ M \in \R{m \times n} ~ | ~ \underline{M} \leq M \leq \overline{M} \}$, where the inequality is evaluated element-wise.
Intervals are a special case of interval matrices, where the lower and upper bounds are vectors.
Additionally, we define an $n$-dimensional hyperball centered at the origin by $\B{\varepsilon} = \big\{x \in \R{n} ~\big|~ \norm{x}_2 \leq \varepsilon \big \} \subset \R{n}$ with respect to the Euclidean norm.
%An interval matrix containing only an all-zero matrix is denoted by $\intmatzeros{}$.
The floor operation $\lfloor x \rfloor$ returns the next smaller integer for a scalar $x$.
For clarify, arguments of functions are sometimes omitted.
Finally, we use $f(t) \sim \bigO{t^a}$ to represent that the function $f(t)$ approaches its limit value ($0$ or infinity in this work) as fast or faster than $t^a$.

% SET REPRESENTATIONS AND OPERATIONS
\vspace{-0.2cm}
\subsection{Definitions}
\label{ssec:definitions}

All sets are assumed to be compact, convex, and bounded.
In this work, we represent outer-approximations of reachable sets with zonotopes \cite[Def.~1]{Girard2005}:
\begin{definition}[Zonotope] \label{def:zonotope}
Given a center vector $c \in \R{n}$ and a generator matrix $G \in \R{n \times \gens{}}$, a zonotope $\Z{} \subset \R{n}$ is
\begin{equation*}
	\Z{} := \bigg\{ c + \sum_{i = 1}^{\gens} G_{(\cdot,i)} \, \alpha_i ~\bigg| ~ \alpha_i \in [-1,1] \bigg\}.
\end{equation*}
The zonotope order is defined as $\zonorder{} := \frac{\gens{}}{n}$ and we use the shorthand $\Z{} = \zono{c,G}$.
\end{definition}
\noindent Moreover, we represent inner-approximations of reachable sets with constrained zonotopes \cite[Def.~3]{Scott2016}:
\begin{definition}[Constrained zonotope] \label{def:conZonotope}
Given a vector $c \in \R{n}$, a generator matrix $G \in \R{n \times \gens{}}$, a constraint matrix $A \in \R{\cons{} \times \gens{}}$, and a constraint offset $b \in \R{\cons{}}$, a constrained zonotope $\CZ{}~\subset~\R{n}$ is
\begin{equation*}
		\CZ{} := \bigg\{ c + \sum_{i = 1}^{\gens{}} G_{(\cdot,i)} \, \alpha_i ~\bigg| ~ \sum_{i=1}^\gens{} A_{(\cdot, i)} \alpha_i = b, ~ \alpha_i \in [-1,1] \bigg\}.
\end{equation*}
We use the shorthand $\CZ{} = \conZono{c,G,A,b}$.
\end{definition}
\noindent Finally, unsafe sets and safe sets are represented by polytopes \cite[Sec.~1.1]{Ziegler2012}:
\begin{definition}[Polytope] \label{def:polytope}
Given a constraint matrix $C \in \R{\polyCons{} \times n}$ and a constraint offset $d \in \R{\polyCons{}}$, the halfspace representation of a polytope $\P{} \subset \R{n}$ is
\begin{equation*}
	\P{} := \big\{ x \in \R{n}~\big|~ C x \leq d \big\}.
\end{equation*}
Equivalently, one can use the vertex representation
\begin{equation*}
	\P{} := \bigg\{ \sum_{i=1}^{\polyVert} \beta_i \, v_i~\bigg|~ \sum_{i=1}^{\polyVert} \beta_i = 1,~ \beta_i \geq 0 \bigg\},
\end{equation*}	
where $\{v_1, \dotsc, v_\polyVert{}\} \in \R{n}$ are the polytope vertices.
We use the shorthands $\P{} = \poly{C,d}$ and $\P{} = \polyV{[v_1 \dotsc v_\polyVert{}]}$.
\end{definition}
Given the sets $\S{1},\S{2} \subset \R{n}, \S{3} \subset \R{m}$ and a matrix $M \in \R{w \times n}$, we require the set operations linear map $M \S{1}$, Cartesian product $\S{1} \times \S{3}$, Minkowski sum $\S{1} \oplus \S{2}$, Minkowski difference $\S{1} \ominus \S{2}$, and linear combination $\linCombOp{\S{1},\S{2}}$, which are defined as 
\begin{align}
	& M \S{1} = \{ M s ~|~ s \in \S{1} \} , \label{eq:defLinTrans} \\
	& \S{1} \times \S{3} = \{ [s_1^\top ~ s_3^\top]^\top ~ | ~ s_1 \in \S{1}, s_3 \in \S{3} \} , \label{eq:defCartProd} \\
	& \S{1} \oplus \S{2} = \{ s_1 + s_2 ~|~ s_1 \in \S{1}, s_2 \in \S{2} \} , \label{eq:defMinSum} \\
	& \S{1} \ominus \S{2} = \{ s~|~ s \oplus \S{2} \subseteq \S{1} \} , \label{eq:defMinDiff} \\
\begin{split} \label{eq:defLinComb}
	& \linCombOp{\S{1},\S{2}} = \big\{\lambda s_1 + (1-\lambda)s_2~| \\
	& \hspace{75pt} s_1 \in \S{1}, s_2 \in \S{2}, \lambda \in [0,1] \big\}.
\end{split}
\end{align}
For zonotopes $\Z{1} = \zono{c_1,G_1},\Z{2} = \zono{c_2,G_2} \subset \R{n}$, linear map and Minkowski sum can be computed as \cite[Eq.~(2.1)]{Althoff2010a}
\begin{align}
	& M \Z{1} = \zono{M c_1, M G_1} , \label{eq:zonoLinTrans} \\
%	& \Z{1} \times \Z{3} = \bigg\langle \begin{bmatrix} c_1 \\ c_3 \end{bmatrix}, \begin{bmatrix} G_1 & \matzeros \\ \matzeros{} & G_3 \end{bmatrix} \bigg\rangle_Z , \label{eq:zonoCartProd} \\ 
	& \Z{1} \oplus \Z{2} = \zono{c_1 + c_2, [G_1~G_2]} . \label{eq:zonoMinSum}
\end{align} 
Since the convex hull represents an enclosure of the linear combination, we furthermore obtain \cite[Eq.~(2.2)]{Althoff2010a}
{\setlength{\belowdisplayskip}{2pt}
\begin{align} \label{eq:zonoLinComb}
	\begin{split}
		& \linCombOp{\Z{1},\Z{2}} \subseteq \big \langle 0.5(c_1 + c_2),~\big[ 0.5(c_1 - c_2) \\
	& \qquad \qquad 0.5 (G_1 + G^{(1)}_2) ~~ 0.5 (G_1 - G_2^{(1)}) ~~ G_2^{(2)}  \big] \big \rangle_Z
	\end{split}
\end{align}}%
with
\begin{equation*}
	G_2^{(1)} = [G_{2(\cdot,1)} \, \dots \, G_{2(\cdot,\gens{}_1)}], ~ G_2^{(2)} = [G_{2(\cdot,\gens{}_1+1)} \, \dots \, G_{2(\cdot,\gens{}_2)}],
\end{equation*}
where we assume without loss of generality that $\Z{2}$ has more generators than $\Z{1}$ so that we can write the formula in a compact form.
Later on in \secref{sec:innerApprox}, we show that the Minkowski difference of two zonotopes can be represented as a constrained zonotope.
The multiplication $\intmat{I} \, \Z{}$ of an interval matrix $\intmat{I}$ with a zonotope $\Z{}$ can be outer-approximated as specified in \cite[Thm.~4]{Althoff2007c}, and the operation $\boxOp{\Z{}}$ returns the axis-aligned box outer-approximation according to \cite[Prop.~2.2]{Althoff2010a}.
The representation size of a zonotope can be decreased with zonotope order reduction.
While our reachability algorithm is compatible with all common reduction techniques \cite{Yang2016}, we focus on Girard's method \cite[Sec.~3.4]{Girard2005} for simplicity:
\vspace{-0.1cm}
\begin{definition}[Zonotope order reduction] \label{def:reduce}
Given a zonotope $\Z{} = \zono{c,G} \subset \R{n}$ and a desired zonotope order $\overzonorder{} \geq 1$, the operation $\redOp{\Z{},\overzonorder{}} \supseteq \Z{}$ returns an enclosing zonotope with order smaller or equal to $\overzonorder{}$:
{\setlength{\belowdisplayskip}{2pt}
\begin{align*}
	\redOp{\Z{},\overzonorder{}} = \zono{c,[\Gkeep{}~\Gred{}]} ,
\end{align*}}%
with
{\setlength{\abovedisplayskip}{2pt}
\begin{equation*}
	\Gkeep{} = [ G_{(\cdot,\pi_{\chi+1})} ~\dots~G_{(\cdot,\pi_\gens{})}],
	\;\;
	\Gred{} = \diag \bigg (\sum_{i=1}^\chi \big | G_{(\cdot,\pi_i)} \big|\bigg),
\end{equation*}}%
where $\pi_1,\dots,\pi_\gens$ are the indices of the sorted generators
\begin{equation*}
	\| G_{(\cdot,\pi_1)} \|_1 - \| G_{(\cdot,\pi_1)} \|_\infty \leq ... \leq \| G_{(\cdot,\pi_\gens{})} \|_1 - \| G_{(\cdot,\pi_\gens{})} \|_\infty,
\end{equation*}
and $\chi = \gens{} - \lfloor (\overzonorder{} -1) n \rfloor$ is the number of reduced generators.
\end{definition}
For distances between sets, we use the Hausdorff distance:
\begin{definition}[Hausdorff distance] \label{def:dH}
For two compact sets $\S{1}, \S{2} \subseteq \R{n}$, the Hausdorff distance with respect to the Euclidean norm is defined as
\begin{align}
\begin{split} \label{eq:dH}
	d_H (\S{1}, \S{2})
	&= \max \Big\{
		\max_{s_1 \in \S{1}}
		\Big( \min_{s_2 \in \S{2}} \, \norm{s_1 - s_2}_2 \Big), \\
	&\qquad \qquad \quad \max_{s_2 \in \S{2}}
		\Big( \min_{s_1 \in \S{1}} \norm{s_1 - s_2}_2 \Big)
	\Big\} .
\end{split}
\end{align}
Using a hyperball $\B{\varepsilon}$ of radius $\varepsilon$, an alternative definition is
\begin{equation} \label{eq:dHball}
	d_H (\S{1}, \S{2}) = \varepsilon
	\, \Leftrightarrow \,
	\S{2} \subseteq \S{1} \oplus \B{\varepsilon} \wedge
	\S{1} \subseteq \S{2} \oplus \B{\varepsilon} .
\end{equation}
\end{definition}
\noindent
Moreover, we will frequently use the operator %\cite[Prop.~1]{Wetzlinger2020}
\begin{equation} \label{eq:errOp}
	\errOp{\S{}} := \rad( \boxOp{\S{}} ) .
\end{equation}
As an immediate consequence of \eqref{eq:errOp}, we obtain
\begin{equation} \label{eq:dH<=err}
	\dH{\matzeros{},\S{}} \overset{\matzeros{} \in \S{}}{\leq} \errOp{\S{}} ,
\end{equation}
which states that the Hausdorff distance between a set and the origin can be bounded by the radius of an enclosing hyperball.

% -----------------------------------------------------------------------

\section{Problem Formulation}
\label{sec:problemformulation}

We consider linear time-invariant systems of the form
\begin{align}
	\dot{x}(t)	&= Ax(t) + Bu(t) + p, \label{eq:linsys} \\
	y(t)		&= Cx(t) + Wv(t) + q, \label{eq:output}
\end{align}
with $A \in \R{n \times n}$, $B \in \R{n \times m}$, $C \in \R{\ell \times n}$, $W \in \R{\ell \times o}$, where $x(t) \in \R{n}$ is the state, $u(t) \in \R{m}$ is the input, $y(t) \in \R{\ell}$ is the output, $v(t) \in \R{o}$ is a measurement error, $p \in \R{n}$ is a constant input, and $q \in \R{\ell}$ is a constant offset on the output.
The initial state $x(t_0)$ is uncertain within the initial set $\initset{} \subset \R{n}$, the input $u(t)$ is uncertain within the input set $\inputset{} \subset \R{m}$, and $v(t)$ is uncertain within the set of measurement errors $\mathcal{V} \subset \R{o}$.
In this work, we assume that $\initset{}$, $\inputset{}$, and $\mathcal{V}$ are represented by zonotopes.
Using $\inputset{} = \zono{c_u,G_u}$, we define the vector $\uTrans{} = B c_u + p \in \R{n}$ and the set $\inputsetzero{} = \zono{\matzeros{}, B G_u} \subset \R{n}$ for later derivations.
Please note that we assume a constant vector $\uTrans{}$ to keep the presentation simple, but the extension to time-varying inputs $\uTrans{}(t)$ is straightforward.
Without loss of generality, we set the initial time to $t_0 = 0$ and the time horizon to $[0,\tFinal{}]$.
The reachable set is defined as follows:

\begin{definition}[Reachable set] \label{def:R}
Let us denote the solution to \eqref{eq:linsys} for the initial state $x(0)$ and the input signal $u(\cdot)$ by $\xi(t;x(0),u(\cdot))$.
Given an initial set $\initset{}$ and an input set $\inputset{}$, the reachable set at time~$t \geq 0$ is
\begin{equation*}
	\Rti{t} := \big\{ \xi(t;x(0),u(\cdot)) ~ \big| ~ x(0) \in \initset{}, \, \forall \theta \in [0,t]\!: u(\theta) \in \inputset{} \big\} .
\end{equation*}
We denote the time-point reachable set at time $t = t_k$ by $\Rti{t_k}$ and the time-interval reachable set over $\tau_k \in [t_k,t_{k+1}]$ by $\Rti{\tau_k} := \bigcup_{t \in [t_k,t_{k+1}]} \Rti{t}$.
\end{definition}

Since the exact reachable set as defined in \defref{def:R} cannot be computed for general linear systems \cite{Gan2018}, we aim to compute tight outer-approximations $\overRtp{t} \supseteq \Rtp{t}$ and inner-approximations $\underRtp{t} \subseteq \Rtp{t}$ instead.

\begin{algorithm}[t]
\caption{Reachability algorithm (manual tuning)} \label{alg:standard}
\textbf{Require:} Linear system $\dot{x} = Ax + Bu + p$,
	initial set~$\initset{} = \zono{c_x,G_x}$, input set~$\inputset{} = \zono{c_u,G_u}$, time horizon~$\tFinal{}$,
	time step size~$\Delta t$ dividing the time horizon into an integer number of steps, truncation order $\tayl{}$, zonotope order~$\zonorder{}$
	
\textbf{Ensure:} Outer-approximation of the reachable set $\overRti{[0,\tFinal{}]}$

\setstretch{1.2}
\begin{algorithmic}[1]
\State $t_0 \gets 0, \uTrans{} \gets B c_u + p$, $\inputsetzero{} \gets \zono{\matzeros{},B G_u}$
\State $\HPtp{t_0} \gets \initset{}, \Pu{t_0} \gets \matzeros{}, \overPU{t_0} \gets \matzeros{}$
\State $\Pu{\Delta t} \gets$ Eq.~\eqref{eq:Pu_init}, $\overPU{\Delta t} \gets$ Eq.~\eqref{eq:overPU_init}
	\label{alg:standard:overPU/Pu_init}
\For{$k \gets 0$ to $\frac{\tFinal{}}{\Delta t}-1$}
	\State $t_{k+1} \gets t_k + \Delta t, \tau_{k+1} \gets [t_k,t_{k+1}]$
	% superposition part 1: affine dynamics (hom. + part. due to u)
%	\State $\overPu{t_{k+1}} \gets \redOp{ \overPu{t_k} \oplus e^{At_k} \overPu{\Delta t}, \zonorder{} }$
%		\label{alg:standard:overPu_k}
	\State $\Pu{t_{k+1}} \gets \Pu{t_k} \oplus e^{At_k} \Pu{\Delta t}$
		\label{alg:standard:Pu_k}
	\State $\HPtp{t_{k+1}} \gets e^{A t_{k+1}} \initset{} + \Pu{t_{k+1}}$
		\label{alg:standard:HPtp_k}
	\State $\C{} \gets \Fx{} \HPtp{t_k} \oplus \Fu{} \uTrans{}$
		\Comment{see \eqref{eq:Fx}-\eqref{eq:Fu}}
		\label{alg:standard:C}
	\State $\overHPti{\tau_k} \gets \linCombOp{\HPtp{t_k}, \HPtp{t_{k+1}}} \oplus \C{}$
		\label{alg:standard:overHPti_k}
%	\Statex $\qquad \quad \oplus \, \Fx{k} \overHPtp{t_k} \oplus \Fu{k} \uTrans{}$
	% superposition part 2: particular solution due to inputset
	\State $\overPU{t_{k+1}} \gets \redOp{ \overPU{t_k} \oplus e^{A t_k} \overPU{\Delta t}, \zonorder{} }$
		\label{alg:standard:overPU_k}
	% combination
	\State $\overRti{\tau_k} \gets \overHPti{\tau_k} \oplus \overPU{t_{k+1}}$
		\label{alg:standard:overRti}
\EndFor
\State $\overRti{[0,\tFinal{}]} \gets \bigcup_{k=0}^{\frac{\tFinal{}}{\Delta t}-1} \overRti{\tau_k}$
	\label{alg:standard:fullR}
\end{algorithmic}
\end{algorithm}

\begin{figure}[b]
	\centering
	\input{./figures/affine3steps.tikz}
	\caption{Visualization of the three steps required to compute the homogeneous solution of the first time interval, adapted from \cite[Fig.~3.1]{Althoff2010a}.}
	\label{fig:affine3steps}
\end{figure}

Our automated tuning approach is based on \algref{alg:standard}, which is a slight modification of the wrapping-free propagation-based reachability algorithm \cite{LeGuernic2010}.
Fundamentally, one exploits the superposition principle of linear systems by separately computing the homogeneous and particular solutions, which are then combined in \lineref{alg:standard:overRti} to yield the overall reachable set for each time interval.
The homogeneous solution can be enclosed using the following three steps visualized in \figref{fig:affine3steps}:
%(cf.~\cite[Fig.~3.1]{Althoff2010a}):
%
%
\begin{enumerate}
	\item Compute the time-point solution by propagating the initial set with the exponential matrix $e^{A \Delta t}$.
	\item Approximate the time-interval solution with the linear combination \eqref{eq:zonoLinComb}, which would only enclose straight-line trajectories.
	\item Account for the curvature of the trajectories by enlarging the linear combination with the set $\C{}$.
\end{enumerate}
The curvature enclosure $\C{}$ is computed in \lineref{alg:standard:C} of \algref{alg:standard} using the interval matrices \cite[Sec.~3.2]{Althoff2010a}
\begin{align}
	&\Fx{} = \bigoplus_{i=2}^{\tayl{}} \mathcal{I}_i(\Delta t) \, \frac{A^i}{i!} \oplus \E{} \label{eq:Fx} \\
	&\Fu{} = \bigoplus_{i=2}^{\tayl{}+1} \mathcal{I}_i(\Delta t) \, \frac{A^{i-1}}{i!} \oplus \E{}\Delta t \label{eq:Fu} \\
	&\text{with} \quad \mathcal{I}_i(\Delta t) = \big[\big( i^{\frac{-i}{i-1}} - i^{\frac{-1}{i-1}} \big) \Delta t^i,0\big] , \label{eq:I}
\end{align}
where the interval matrix $\E{k}$ represents the remainder of the exponential matrix \cite[Eq.~(3.2)]{Althoff2010a}:
\begin{align}
\begin{split} \label{eq:E}
	\E{} &= [-E(\Delta t,\tayl{}),E(\Delta t,\tayl{})] , \\
	E(\Delta t, \tayl{}) &= e^{|A|\Delta t} - \sum_{i=0}^{\tayl{}} \frac{\big( |A|\Delta t \big)^i}{i!} .
\end{split}
\end{align}
\noindent To increase the tightness, the particular solution $\Pu{\Delta t}$ due to the constant input $\uTrans{}$ \cite[Eq.~(3.7)]{Althoff2010a},
\begin{equation} \label{eq:Pu_init}
	\Pu{\Delta t} = A^{-1} (e^{A\Delta t} - I_n) \, \uTrans{} ,
\end{equation}
is added to the homogeneous solution in \lineref{alg:standard:HPtp_k}.
If the matrix $A$ is not invertible, we can integrate $A^{-1}$ in the power series of the exponential matrix to compute $\Pu{\Delta t}$.
The particular solution due to the time-varying input within the set $\inputsetzero{}$ can be enclosed by \cite[Eq.~(3.7)]{Althoff2010a}
\begin{equation} \label{eq:overPU_init}
	\overPU{\Delta t}
	= \bigoplus_{i=0}^{\tayl{}} \frac{A^i \Delta t^{i+1}}{(i+1)!} \, \inputsetzero{}
		\oplus \E{} \Delta t \, \inputsetzero{} .
\end{equation}
Finally, the reachable set $\overRti{[0,\tFinal{}]}$ for the entire time horizon is given by the union of the sets for individual time-interval reachable sets according to \lineref{alg:standard:fullR}.

As for all other state-of-the-art reachability algorithms  \cite{Girard2005,Girard2006,LeGuernic2010,Althoff2010a,Frehse2011,Wetzlinger2020}, the main disadvantage of \algref{alg:standard} is that the tightness of the computed reachable set $\overRtp{t}$ is unknown and heavily depends on the chosen time step size $\Delta t$, truncation order $\tayl{}$, and zonotope order $\zonorder{}$.
In this work, we solve both issues by automatically tuning these algorithm parameters such that the Hausdorff distance between the computed enclosure $\overRtp{t}$ and the exact reachable set $\Rtp{t}$ remains below a desired threshold $\emax{}$ at all times:
\begin{equation*} %\label{eq:objective}
	\text{Tune~} \Delta t, \tayl{}, \zonorder{} \quad \text{s.t.} \quad
	\forall t \in [0,\tFinal{}]: \dH{\Rtp{t},\overRtp{t}} \leq \emax{} .
\end{equation*}
We will further utilize this result to efficiently extract an inner-approximation of the reachable set (\secref{sec:innerApprox}) and construct a fully automated verification algorithm (\secref{sec:verification}).

%% file: figures/affine3steps.tikz
\begin{tikzpicture}[scale = 1]

% grid for help
%\draw[gray,thin,step=0.5] (0,0) grid (8.5,5.0);

% start set
\filldraw[fill=good_gray,draw=good_blue,thick] (1,0.5) -- (0.5,1.25) -- (1,2) -- (2,1.5) -- (1.75,0.75) -- cycle;
\node[align=center] at (1.25,1.25) {$\initset{}$};
% end set
\filldraw[fill=good_gray,draw=good_blue,thick] (1.5,2.7) -- (1.25,3.7) -- (2,4.2) -- (2.75,3.45) -- (2.25,2.7) -- cycle;
\node[align=center] at (2,3.4) {$e^{A\Delta t} \initset{}$};

% linear combination
\begin{scope}[xshift=2.75cm]
\filldraw[fill=good_gray,draw=good_blue,thick] (1,0.5) -- (0.5,1.25) -- %% from start set
	(1.25,3.7) -- (2,4.2) -- (2.75,3.45) -- %% from end set
	(1.75,0.75) -- cycle; %% from start set
\draw[good_blue] (1,0.5) -- (0.5,1.25) -- (1,2) -- (2,1.5) -- (1.75,0.75) -- cycle;
\draw[good_blue] (1.5,2.7) -- (1.25,3.7) -- (2,4.2) -- (2.75,3.45) -- (2.25,2.7) -- cycle;
%\node[align=center,rotate=60] at (1.55,2.55) {$\linCombOp{\HPtp{t_k},\HPtp{t_{k+1}}}$};
\node[align=center] at (1.6,2.3) {$\auxlinComb(\cdot)$};
\end{scope}

% enlargement
\begin{scope}[xshift=5.5cm]
\filldraw[fill=good_gray,draw=good_blue,thick] (0.9,0.3) -- (0.3,1.3) -- %% from start set
	(1.05,3.75) -- (2.05,4.45) -- (3,3.45) -- %% from end set
	(2,0.7) -- cycle; %% from start set
\draw[good_blue] (1,0.5) -- (0.5,1.25) -- %% from start set
	(1.25,3.7) -- (2,4.2) -- (2.75,3.45) -- %% from end set
	(1.75,0.75) -- cycle; %% from start set
\draw[good_blue] (0.5,1.25) -- (1,2) -- (2,1.5);
\draw[good_blue] (2.75,3.45) -- (2.25,2.7) -- (1.5,2.7) -- (1.25,3.7);
%\node[align=center,rotate=60] at (1.6,2.55) {$\linCombOp{\HPtp{t_k},\HPtp{t_{k+1}}}$ \\ $\oplus \, \C{}$};
\node[align=center] at (1.625,2.3) {$\auxlinComb(\cdot) \oplus \C{}$};
\end{scope}

% numbers
\node at (0.75,4) {1)};
\node at (3.5,4) {2)};
\node at (6.25,4) {3)};

\end{tikzpicture}

%% file: subtex/tuning.tex
\section{Automated Parameter Tuning}
\label{sec:tuning}

Let us now present our approach for the automated tuning of algorithm parameters.
While \algref{alg:standard} uses fixed values for $\Delta t$, $\tayl{}$, and $\zonorder{}$, we tune different values $\Delta t_k$, $\tayl{k}$, and $\zonorder{k}$ in each step~$k$ based on the induced outer-approximation error in order to satisfy the error bound at all times.
To achieve this, we first derive closed-form expressions describing how the individual errors depend on the values of each parameter in \secref{ssec:errors}. 
Next, we present our automated parameter tuning algorithm in \secref{ssec:automated} and prove its convergence in \secref{ssec:proof}.
Finally, we discuss further improvements to the algorithm in \secref{ssec:tuningmethods} and describe the extension to output sets in \secref{ssec:outputset}.

\subsection{Error Measures}
\label{ssec:errors}

\noindent Several sources of outer-approximation errors exist in \algref{alg:standard}:
\begin{enumerate}
	\item \textbf{Affine dynamics (\lineref{alg:standard:overHPti_k})}:
		The time-interval solution $\overHPti{\tau_k}$ of the affine dynamics
		contains errors originating from enclosing the linear combination by zonotopes
		and the set~$\C{}$ accounting for the curvature of trajectories.
	\item \textbf{Particular solution (\linesref{alg:standard:overPU_k}{alg:standard:overRti})}:
	    Using the outer-approximation $\overPU{t_k}$ based on \eqref{eq:overPU_init} for the particular solution
		due to the input set $\inputsetzero{}$ induces another error.
		Moreover, the Minkowski addition of $\overPU{t_{k+1}}$ to $\overHPti{\tau_k}$
		is outer-approximative as it ignores dependencies in time.
	\item \textbf{Zonotope order reduction (\lineref{alg:standard:overPU_k})}:
		The representation size of the particular solution $\overPU{t_k}$
		has to be reduced, which induces another error.
\end{enumerate}
\noindent In this subsection, we derive upper bounds for all these errors in terms of the Hausdorff distance between an exact set $\S{}$ and the computed outer-approximation $\widehat{\mathcal{S}}$.
We will use two different albeit related error notations:
New errors induced in step~$k$ are denoted by $\Delta \varepsilon^*_k(\Delta t_k, \tayl{k})$ and $\Delta \varepsilon^*_k(\zonorder{k})$ to emphasize the dependence on the respective algorithm parameters (using $*$ as a placeholder for various superscripts). %where we omit the dependence on $\tayl{k}$ for a concise notation.
Some errors for single time steps add up over time.
We accumulate all previous errors $\varepsilon^*_k(\cdot)$ until time $t_k$ by
{\setlength{\abovedisplayskip}{4pt} %\setlength{\belowdisplayskip}{3pt}
\begin{equation} \label{eq:eaccprop}
	\varepsilon^*_{k+1} = \varepsilon^*_k + \Delta \varepsilon^*_k(\cdot) , \quad \varepsilon^*_0 := 0 .
\end{equation}}%
%
%accumulates all previous errors $\varepsilon^*_k(\cdot)$ until time $t_k$,
%where we omit the parameter dependency of previous steps.
Let us now derive closed-form expressions for all errors.

\subsubsection{Affine Dynamics}
\label{par:affine}

We start by determining the error contained in the computed outer-approximation $\overHPti{\tau_k}$ of the time-interval solution for the affine dynamics $\dot{x}(t) = Ax(t) + \uTrans{}$:

%\vspace{-5pt}
\begin{proposition}[Affine dynamics error] \label{prop:eps_affine}
Given the set $\HPtp{t_k} = \zono{c_h,G_h}$ with $G_h \in \R{n \times \gens{}_h}$, the Hausdorff distance between the exact time-interval solution of the affine dynamics
{\setlength{\abovedisplayskip}{2pt}
\begin{align*}
	\HPti{\tau_k} &= \big\{ e^{At} x(t_k) + A^{-1}(e^{A (t - t_k)} - I_n) \uTrans{} ~ \big| \\
		&\qquad \quad t \in \tau_k, \, x(t_k) \in \HPtp{t_k} \big\}
\end{align*}}%
and the corresponding outer-approximation
\begin{equation} \label{eq:overHPti}
	\overHPti{\tau_k} = \linCombOp{\HPtp{t_k}, \HPtp{t_{k+1}}} \oplus \C{} \\
\end{equation}
in \lineref{alg:standard:overHPti_k} of \algref{alg:standard} is bounded by
\begin{align}
	& \dH{\HPti{\tau_k},\overHPti{\tau_k}} \nonumber \\
	& \hspace{12pt} \leq \ehom{k} := 2 \errOp{ \C{} } + \sqrt{\gens{}_h} \, \norm{G_h^{(-)}}_2 , \label{eq:ehom}
\end{align}
where $G^{(-)}_h = (e^{A\Delta t_k} - I_n) G_h$.
\end{proposition}
\begin{proof}
An inner-approximation of $\HPti{\tau_k}$ is given by 
\begin{equation*}
	\innerHti{\tau_k} = \linCombOp{\HPtp{t_k},\HPtp{t_{k+1}}} \ominus \mathcal{B}_\mu \ominus \C{} \subseteq \HPti{\tau_k} ,
\end{equation*}
where we additionally have to subtract a hyperball of radius $\mu = \sqrt{\gens{}_h} \, \big\| G_h^{(-)} \big\|_2$ bounding the Hausdorff distance between the exact linear combination \eqref{eq:defLinComb} and the zonotope enclosure \eqref{eq:zonoLinComb} according to \propref{prop:ecomb} in Appendix~A.
The error $\ehom{k}$ bounds the distance between $\innerHti{\tau_k}$ and the outer-approximation $\overHPti{\tau_k}$ in \eqref{eq:overHPti}. 
\end{proof}
% old formulation:
%According to \propref{prop:ecomb} in Appendix~A, the Hausdorff distance between the exact linear combination \eqref{eq:defLinComb} and the zonotope enclosure \eqref{eq:zonoLinComb} is bounded by $\mu = \sqrt{\gens{}_h} \, \big\| G_h^{(-)} \big\|_2$.
%Hence, we rewrite \eqref{eq:overHPti} to obtain the inner-approximation
%%
%\begin{equation*}
%	\linCombOp{\HPtp{t_k},\HPtp{t_{k+1}}} \ominus \mathcal{B}_\mu \ominus \C{} \subseteq \HPti{\tau_k} .
%\end{equation*}
%%
%The error $\ehom{k}$ then bounds the distance between this inner-approximation and the outer-approximation $\overHPti{\tau_k}$ in \eqref{eq:overHPti}. 

\subsubsection{Particular Solution}
\label{par:PU}

We first account for the error induced by enclosing the exact time-point solution $\PU{t_{k+1}}$ with the outer-approximation $\overPU{t_{k+1}}$:
\begin{proposition}[Time-point error in particular solution] \label{prop:eaccU}
The Hausdorff distance between the exact particular solution
\begin{equation*}
	\PU{t_{k+1}} = \bigg\{ \int_0^{t_{k+1}} e^{A(t_{k+1}-\theta)} u(\theta) \, \mathrm{d}\theta ~ \bigg| ~ u(\theta) \in \inputsetzero{} \bigg\}
\end{equation*}
and the recursively computed outer-approximation
\begin{equation*}
	\overPU{t_{k+1}} = \overPU{t_k} \oplus e^{At_k} \overPU{\Delta t_k}
\end{equation*}
with $\overPU{\Delta t_k}$ computed according to \eqref{eq:overPU_init} is bounded by
\begin{equation} \label{eq:eaccU}
	\eaccU{k+1} = \eaccU{k} + \eaccUstep{k} ,
\end{equation}
where the error $\eaccUstep{k}$ for one time step is bounded by
\begin{align}
\begin{split} \label{eq:eaccUstep}
	&\eaccUstep{k} \\
	&\quad := \errOp{ \hspace{-1pt} e^{At_k} \bigg( \hspace{-2pt} \Big( \sum_{i=1}^{\tayl{k}} \tilde{A}_i \Big) \inputsetzero{}
		\oplus \E{k} \Delta t_k \, \inputsetzero{} \bigg) \hspace{-3pt} } \\
	&\qquad\; + \errOp{ \hspace{-1pt} e^{At_k} \bigg( \bigoplus_{i=1}^{\tayl{k}} \tilde{A}_i \, \inputsetzero{}
		\oplus \E{k} \Delta t_k \, \inputsetzero{} \bigg) \hspace{-3pt} }
\end{split}
\end{align}
with $\tilde{A}_i = \frac{A^i \Delta t_k^{i+1}}{(i+1)!}$ and $\E{k}$ from \eqref{eq:E}.
\end{proposition}

\begin{proof}
The error $\eaccUstep{k}$ for one time step is given by the Hausdorff distance between the computed outer-approximation and an inner-approximation obtained by considering constant inputs according to \propref{prop:errUoneStep} in Appendix~A.
The overall error $\eaccU{k+1}$ follows by error propagation as in \eqref{eq:eaccprop}.
\end{proof}

% previous position of figure
%\begin{figure}[t]
%	\centering
%	\input{./figures/eUstep.tikz}
%	\caption{Reachable set computation using the correct particular time-interval solution $\overPU{\tau_k}$ (left) and an outer-approximation $\overPU{t_{k+1}} \supseteq \overPU{\tau_k}$ (right).}
%	\label{fig:eUstep}
%	\vspace{-0.2cm}
%\end{figure}

\noindent Since $0 \in \inputsetzero{}$, the added set due to uncertain inputs $e^{At_k} \overPU{\theta}$ is equal to $\matzeros{}$ at the beginning of each time step ($\theta = 0$) and monotonically grows towards the set $e^{At_k} \overPU{\Delta t_k}$ at the end of the time step ($\theta = \Delta t_k$) as shown in \figref{fig:eUstep} on the left.
The Minkowski sum in \lineref{alg:standard:overHPti_k} of \algref{alg:standard} ignores this dependency on time, inducing another outer-approximation error:
\begin{proposition}[Time-interval error in particular solution] \label{prop:eaccUstep}
The maximum Hausdorff distance at any time $t \in \tau_k = [t_k,t_{k+1}]$ between the exact time-interval particular solution
\begin{equation*}
	\PU{t} = \bigg\{ \int_{0}^{t} e^{A(t - \theta)} u(\theta) \mathrm{d}\theta ~ \bigg| ~ u(\theta) \in \inputsetzero{} \bigg\}
\end{equation*}
and the outer-approximation $\forall t \in \tau_k: \PU{t} \subseteq \overPU{t_{k+1}}$ is bounded by
\begin{equation*}
	\max_{t \in [t_k,t_{k+1}]} \dH{\PU{t},\overPU{t_{k+1}}} \leq \eaccU{k} + \eUstep{k} \\
\end{equation*}
with $\eaccU{k}$ from \propref{prop:eaccU} and the additional error
\begin{equation} \label{eq:eUstep}
	\eUstep{k} := \errOp{e^{At_k} \overPU{\Delta t_k}} .
\end{equation}
\end{proposition}
\begin{proof}
Since $e^{At_k} \overPU{\theta}$ grows monotonically with $\theta$, the maximum deviation over the time interval $\tau_k$ occurs at $t = t_k$, where the actual additional set would be $\matzeros{}$, but instead $e^{At_k} \overPU{\Delta t_k}$ is used.
Therefore, the error is
\begin{equation*}
	\dH{\matzeros{}, e^{At_k} \overPU{\Delta t_k}} \overset{\eqref{eq:dH<=err}}{\leq} \errOp{ e^{At_k} \overPU{\Delta t_k} } ,
\end{equation*}
which corresponds to the size of the additional set.
% arithmetic proof:
%We rewrite
%%
%\begin{align*}
%	&\max_{t \in [t_k,t_{k+1}]} \dH{\PU{t},\overPU{t_{k+1}}} \\
%	&\quad \overset{\eqref{eq:overPU_k}}{=} \max_{t \in [0,\Delta t_k]} d_H\bigg( \PU{t_k} \oplus e^{At_k} \PU{t}, \\
%	&\qquad \qquad \overPU{t_k} \oplus e^{At_k} \overPU{\Delta t_k} \bigg) \\
%	&\quad\, \leq \dH{\PU{t_k},\overPU{t_k}} \\
%	&\qquad \quad + \max_{t \in [0,\Delta t_k]} \dH{e^{At_k} \PU{t}, e^{At_k }\overPU{\Delta t_k}} ,
%\end{align*}
%%
%where the split into two terms is facilitated by $\matzeros{} \subseteq \PU{t_k} \subseteq \overPU{t_k}$ and $\forall t \in [0,\Delta t_k]: \matzeros{} \subseteq e^{At_k} \PU{t} \subseteq e^{At_k} \overPU{\Delta t_k}$.
%Then, the first term evaluates to $\eaccU{k}$ by \propref{prop:eaccU}, whereas the second term attains its maximum at $t=0$ using $e^{At_k} \overPU{t} = \matzeros{}$ and results in $\eUstep{k}$ by \eqref{eq:dH<=err}, which yields the claim.
\end{proof}

\vspace{-0.1cm}

\subsubsection{Zonotope Order Reduction}
\label{par:reduce}

The zonotope order reduction of the particular solution $\overPU{t_{k+1}}$ in \lineref{alg:standard:overPU_k} of \algref{alg:standard} induces another error.
%We can exploit that large parts of the reduction error for $\overPU{t_{k+1}}$ are already contained in the error $\eaccUstep{k}$:
To determine this reduction error, we first split the particular solution $e^{At_k} \overPU{\Delta t_k}$ into two parts
\begin{align}
	&e^{At_k} \overPU{\Delta t_k} \overset{\eqref{eq:overPU_init}}{=} e^{At_k} \bigg(
		\bigoplus_{i=0}^{\tayl{}} \tilde{A}_i \, \inputsetzero{} \oplus \E{k} \Delta t_k \, \inputsetzero{} \bigg) \nonumber \\
	&\; = e^{At_k} \underbrace{\Delta t_k \, \inputsetzero{}}_{=:\,\overPUzero{\Delta t_k}}
		\oplus \, e^{At_k} \underbrace{ \bigg( \bigoplus_{i=1}^{\tayl{}} \tilde{A}_i \, \inputsetzero{} \oplus \E{k} \Delta t \, \inputsetzero{} \bigg)}_{=:\,\overPUrest{\Delta t_k}} \label{eq:overPUzerorest}
\end{align}
with $\tilde{A}_i$ defined as in \propref{prop:eaccU}.
We exploit that the error $\eaccUstep{k}$ in \eqref{eq:eaccUstep} is unaffected by using the box outer-approximation $\operatorname{box} (e^{At_k} \overPUrest{\Delta t_k})$ instead of $e^{At_k} \overPUrest{\Delta t_k}$ since the box outer-approximation is also used in the computation of $\eaccUstep{k}$.
Therefore, we can always reduce $\overPUrest{t_{k+1}}$ to a box (which has zonotope order~$1$) as that reduction error is already contained in $\eaccUstep{k}$.
Consequently, we only have to determine the reduction error of $\overPUzero{t_{k+1}}$:
%In contrast to \algref{alg:standard}, we therefore enclose this part by a box since the resulting error in our adaptive algorithm is equal to using the original zonotope and the computation is faster as we reduce almost all generators in \eqref{eq:overPUzerorest}.

\begin{proposition}[Zonotope order reduction error] \label{prop:zonorder}
The Hausdorff distance between the particular solution $\overPUzero{t_{k+1}}$ computed without any reduction and its iteratively reduced counterpart $\redOp{\overPUzero{t_{k+1}},\zonorder{k}}$ is bounded by
\begin{equation} \label{eq:ered}
	\ered{k+1} = \ered{k} + \eredstep{k} ,
\end{equation}
where the error $\eredstep{k}$ for one time step is bounded by
\begin{align}
\begin{split} \label{eq:eredstep}
	&\dH{ \overPUzero{t_{k+1}}, \redOp{\overPUzero{t_{k+1}},\zonorder{k}} } \\
	&\quad \leq \eredstep{k} := \errOp{\zono{\matzeros{},\Gred{}}} ,
\end{split}
\end{align}
where $\Gred{}$ defined as in \defref{def:reduce} contains the generators selected for reduction.
\end{proposition}
\begin{proof}
The error $\eredstep{k}$ for one time step is given by the box enclosure of the zonotope formed by the generators selected for reduction.
Using the error propagation formula \eqref{eq:eaccprop}, we then obtain the overall error $\ered{k+1}$ in \eqref{eq:ered}.
% more rigorous version:
%The order reduction can be expressed as
%%
%\begin{equation*}
%	\auxred\big( \underbrace{ \zono{c,G_1} \oplus \zono{\matzeros{},G_2} }_{\overPU{t_k}}, \zonorder{} \big)
%	\Leftrightarrow
%	\zono{c,G_1} \oplus \mathcal{I} 
%\end{equation*}
%% note: use auxred instead of redOp here because parentheses get very large otherwise
%with $\mathcal{I} = \boxOp{\zono{\matzeros{},G_2}}$.
%Using \cite[Lemma~1]{Wetzlinger2022} \rmk{from NAHS}, we have
%%
%\begin{align*}
%	\dH{\Z{},\Zover{}}
%	&= \dH{\zono{c,G_1} \oplus \zono{\matzeros{},G_2}, \zono{c,G_1} \oplus \mathcal{I} } \\
%	&\leq \dH{ \zono{\matzeros{},G_2}, \boxOp{\zono{\matzeros{},G_2}} } \\
%	&\leq \errOp{ \zono{\matzeros{},G_2} } ,
%\end{align*}
%where the last inequality is shown in \cite[Eq.~(13)]{Wetzlinger2022}. \rmk{cite is from NAHS}
%\rmk{cite whole proof from NAHS?}
\end{proof}
%
%Note that we can always bring the reduction error to zero by simply omitting the reduction, which corresponds to setting the desired zonotope order of the current step $\zonorder{k}$ to the order of the zonotope $\overPU{t_k}$.

% alternative position of the figure
\begin{figure}[t]
	\centering
	\input{./figures/eUstep.tikz}
	\caption{Reachable set computation using the correct particular time-interval solution $\overPU{\tau_k}$ (left) and an outer-approximation $\overPU{t_{k+1}} \supseteq \overPU{\tau_k}$ (right).}
	\label{fig:eUstep}
	\vspace{-0.2cm}
\end{figure}

\subsubsection{Summary}
\label{par:eps_R}

The derived error terms allow us to compute an upper bound for the outer-approximation error contained in the time-point solution and time-interval solution:
\begin{align}
	&\dH{\Rtp{t_k},\overRtp{t_k}} \leq \eaccU{k} + \ered{k} , \label{eq:eps_Rtp} \\
	\begin{split} \label{eq:eps_Rti}
	&\dH{\Rti{\tau_k},\overRti{\tau_k}} \leq \etotalx{k} \\
	&\quad := \ehom{k} + \eaccU{k} + \eUstep{k} + \ered{k+1} ,
	\end{split}
\end{align}
where we use $\eaccU{k}$ instead of $\eaccU{k+1}$ in \eqref{eq:eps_Rti}, since the difference $\eaccU{k+1} - \eaccU{k} \overset{\eqref{eq:eaccU}}{=} \eaccUstep{k}$ is already included in $\eUstep{k}$.
%Note that we use $\eaccU{k}$ e_^U_k instead of $\eaccU{k+1}$ e_^U_k+1 for the time interval error since $\eaccU{k}$ e_^U_k+1 is already included in $\eUstep{k}$ e_k^U,\tau.
As usually only time-interval reachable sets are required for formal verification, we will only use the time-interval error $\etotalx{k}$ in our automated parameter tuning algorithm.

\subsection{Automated Tuning Algorithm}
\label{ssec:automated}

Using the error terms derived in the previous subsection, we now present an algorithm that tunes $\Delta t_k$, $\tayl{k}$, and $\zonorder{k}$ automatically such that the Hausdorff distance between the exact reachable set $\Rti{[0,\tFinal{}]}$ and the computed enclosure $\overRti{[0,\tFinal{}]}$ is below the error bound $\emax{}$.
As different types of errors require different strategies for parameter tuning, we divide the derived errors into three categories:
\begin{enumerate}
	\item \textbf{Non-accumulating error} $\enonacc{k}$:
		Since the errors $\ehom{k}$ and $\eUstep{k}$ only affect the current step, we define the non-accumulating error by
		\begin{equation} \label{eq:enonacc}
			\enonacc{k} := \ehom{k} + \eUstep{k} .
		\end{equation}				
%		as the non-accumulating error.
	\item \textbf{Accumulating error} $\eacc{k}$:
		The particular solution $\overPU{t_k}$	accumulates over time, yielding
		\begin{equation} \label{eq:eacc_eaccstep}
			\eacc{k} := \eaccU{k}, \qquad \eaccstep{k} := \eaccUstep{k} ,
		\end{equation}
		for the overall accumulating error and the accumulating error for one time step.
	\item \textbf{Reduction error} $\ered{k}$:
		The representation size of the particular solution $\overPU{t_k}$ is iteratively reduced (\lineref{alg:standard:overPU_k}),
		which induces an accumulating error \eqref{eq:ered}.
		Despite its accumulation, we do not add this error to $\eacc{k}$ since it does not directly depend on the time step size $\Delta t_k$.
\end{enumerate}

We have to manage these errors over time so that the resulting set $\overRti{t}$ respects the error bound $\emax{}$ at all times.
Therefore, we partition $\emax{}$ into individual admissible errors $\enonaccadm{k}, \eaccadmstep{k}$, and $\eredadmstep{k}$ for each step, which is visualized in \figref{fig:errors}:

\begin{figure}[t]
	\centering
	\input{./figures/errors.tikz}
	\caption{The errors $\eacc{k}$ and $\ered{k}$ until $t_k$ and the bounds $\eaccadm{t}$ and $\eredadm{t}$ yield the individual error bounds $\enonaccadm{k}, \eaccadmstep{k}, \eredadmstep{k}$ for the current step~$k$.}
	\label{fig:errors}
\end{figure}

\begin{enumerate}
	\item \textbf{Reduction error bound} $\eredadmstep{k}$:
	We limit the reduction error by a linearly increasing bound
	\begin{equation} \label{eq:eredadm}
		\eredadm{t} = \frac{t}{\tFinal{}} \, \zeta \emax{}, \qquad \zeta \in [0,1) .
	\end{equation}
	Thus, the additional error $\eredstep{k}$ in step~$k$ is bounded by
	\begin{equation} \label{eq:eredadmstep}
		\eredadmstep{k} = \eredadm{t_k + \Delta t_k} - \ered{k} ,
	\end{equation}
	i.e., the difference between the bound at time $t_k + \Delta t_k$ and the accumulated error until $t_k$.
	While our algorithm works for arbitrary values $\zeta$, we present a heuristic for choosing $\zeta$ later in \secref{ssec:tuningmethods}.
	\item \textbf{Accumulating error bound} $\eaccadmstep{k}$:
	Similarly, we limit the accumulating error by another linearly increasing bound
	\begin{equation} \label{eq:eaccadm}
		\eaccadm{t} = \frac{t}{\tFinal{}} \, (1-\zeta) \emax{} ,
	\end{equation}
	so that we have $\eredadm{\tFinal{}} + \eaccadm{\tFinal{}} = \emax{}$.
	Analogously to \eqref{eq:eredadmstep}, the bound for the additional error $\eaccUstep{k}$ is
	\begin{equation} \label{eq:eaccadmstep}
		\eaccadmstep{k} = \eaccadm{t_k + \Delta t_k} - \eacc{k} .
	\end{equation}
%	so that the error $\eaccUstep{k}$ has to satisfy $\eaccUstep{k} \leq \eaccadmstep{k}$.
	\item \textbf{Non-accumulating error bound} $\enonaccadm{k}$:
	Finally, we obtain the bound for the non-accumulating error $\enonacc{k}$ by subtracting the other two bounds from $\emax{}$:
	\begin{equation} \label{eq:enonaccadm}
		\enonaccadm{k} = \emax{} - \eredadm{t_k + \Delta t_k} - \eacc{k} .
	\end{equation}
	Note that we only subtract $\eacc{k}$ instead of $\eaccadm{t_k + \Delta t_k}$ for the accumulating error
	since the accumulating error $\eaccstep{k} = \eaccUstep{k}$ for the current step is already accounted for by the error $\eUstep{k}$,
	which is according to \eqref{eq:enonacc} part of the non-accumulating error.
	This also guarantees us a non-zero bound for $\enonaccadm{k}$ in the last step even though $\eaccadm{\tFinal{}} + \eredadm{\tFinal{}} = \emax{}$.
\end{enumerate}

The tuning strategies for the parameters are as follows:
\begin{itemize}
	\item \textbf{Time step size} $\Delta t_k$:
		We initialize $\Delta t_k$ by its previous value $\Delta t_{k-1}$,
		or by $\tFinal{}$ as an initial guess for the first step.
		To keep the presentation simple, we iteratively halve this value until the error bounds are satisfied;
		a more sophisticated tuning method is described in \secref{ssec:tuningmethods}.
	\item \textbf{Truncation order $\tayl{k}$}:
		We tune $\tayl{k}$ simultaneously with the computation of $\Fx{k}$ and $\Fu{k}$,
		for which the idea proposed in \cite[Sec.~3.1]{Wetzlinger2021} is reused:
		The partial sums
		{\setlength{\abovedisplayskip}{2pt}
		\begin{equation} \label{eq:T}
			\T{j} = \bigoplus_{i=1}^{j} \mathcal{I}_i \frac{A^i}{i!}
		\end{equation}}%
		in the computation of $\Fx{k}$ in \eqref{eq:Fx} are successively compared until the relative change
		in the Frobenius norm of $\T{j}$ computed according to \cite[Thm.~10]{Farhadsefat2011} is smaller than $10^{-10}$.
		As this bound is relative, we can ensure convergence independently of the scale of the system,
		since the size of the additional terms decreases exponentially for $i \to \infty$.
	\item \textbf{Zonotope order $\zonorder{k}$}:
		We iteratively increase the order $\zonorder{k}$ until the error $\eredstep{k}$
		is smaller than the error bound $\eredadmstep{k}$.
		A more efficient method compared to this naive implementation is to directly integrate the search
		for a suitable order into the zonotope order reduction.
\end{itemize}

\begin{algorithm}[!thb]
\caption{Reachability algorithm (automated tuning)} \label{alg:automated}
\textbf{Require:} Linear system $\dot{x} = Ax + Bu + p$,
	initial set~$\initset{} = \zono{c_x,G_x}$, input set~$\inputset{} = \zono{c_u,G_u}$, time horizon~$\tFinal{}$,
	error bound~$\emax{}$
	
\textbf{Ensure:} Outer-approximation of the reachable set $\overRti{[0,\tFinal{}]}$

\setstretch{1.2}
\begin{algorithmic}[1]
\State $k \gets 0, t_0 \gets 0, \Delta t_{-1} \gets \tFinal{}, \HPtp{t_0} \gets \initset{}$
\State $\Pu{t_0} \hspace{-1pt} \gets \hspace{-1pt} \zono{\mathbf{0},[\;\!]}, \overPUzero{t_0} \hspace{-1pt} \gets \hspace{-1pt} \zono{\mathbf{0},[\;\!]}, \overPUrest{t_0} \hspace{-1pt} \gets \hspace{-1pt} \zono{\mathbf{0},[\;\!]}$
\State $\uTrans{} \gets B c_u + p, \, \inputsetzero{} \gets \zono{\matzeros{},B G_u}$ %\overPUzero{t_0} \gets \matzeros{}, \overPUrest{t_0} \gets \matzeros{}$
	\label{alg:automated:init}
\While{$t_k < \tFinal{}$} \label{alg:automated:looptstart}
	\State $\Delta t_k \gets 2 \Delta t_{k-1}$
		\label{alg:automated:initDeltat}
	\Repeat \label{alg:automated:loopstart}
		% shrink time step size
		\State $\Delta t_k \gets \frac{1}{2} \Delta t_k$, $t_{k+1} \gets t_k + \Delta t_k$
			\label{alg:automated:shrinkDeltat}
		\State $\tayl{k} \gets 0, \T{\tayl{k}} = \matzeros{}$
		% non-accumulating errors
		\Repeat \label{alg:automated:tayl}
			\State $\tayl{k} \gets \tayl{k} + 1$
				\label{alg:automated:incrementeta}
			\State $\T{\tayl{k}} \gets \T{\tayl{k}-1} \oplus \mathcal{I}_{\tayl{k}} \frac{A^{\tayl{k}}}{\tayl{k}!}$
				\Comment{see~\eqref{eq:I}, \hspace{-1pt}\eqref{eq:T}}
				\label{alg:autoamted:tempeta}
		\Until{$1 - \norm{\T{\tayl{k}-1}}_F / \norm{\T{\tayl{k}}}_F \leq 10^{-10}$}
			\label{alg:automated:tayl_end}
		% all errors
		\State $\enonacc{k}, \eaccstep{k} \gets$ \eqref{eq:enonacc}, \eqref{eq:eacc_eaccstep}
			\label{alg:automated:allerrs}
		% all error bounds
		\State $\eaccadmstep{k}, \enonaccadm{k} \gets$ \eqref{eq:eaccadmstep}, \eqref{eq:enonaccadm}
			\label{alg:automated:allbounds}
	% check errors
	\Until{$\eaccstep{k} \leq \eaccadmstep{k} \wedge \enonacc{k} \leq \enonaccadm{k}$}
		\label{alg:automated:errcheck}
	% set propagation
	\State $\overPUzero{\Delta t_k}, \overPUrest{\Delta t_k} \gets$ \eqref{eq:overPUzerorest}
		\label{alg:automated:overPU_init}
	\State $\overPUrest{t_{k+1}} \gets \overPUrest{t_k} \oplus \boxOp{ e^{A t_k} \overPUrest{\Delta t_k} }$
		\label{alg:automated:overPUrest_k}
	\State $\overPUzero{t_{k+1}} \gets \overPUzero{t_k} \oplus e^{A t_k} \overPUzero{\Delta t_k}$
		\label{alg:automated:overPUzero_k}
	% reduction error
	\State $\eredadmstep{k} \gets $ \eqref{eq:eredadmstep}, $\zonorder{k} \gets 0$
		\label{alg:automated:zonorder_init}
	\Repeat \label{alg:automated:zonorder}
		\State $\zonorder{k} \gets \zonorder{k} + 1$, $\eredstep{k} \gets $ \propref{prop:zonorder}
			\label{alg:automated:incrementzonorder}
	\Until{$\eredstep{k} \leq \eredadmstep{k}$}
		\label{alg:automated:zonorder_end}
	\State $\overPU{t_{k+1}} \gets \redOp{\overPUzero{t_{k+1}},\zonorder{k}} \oplus \overPUrest{t_{k+1}}$
		\label{alg:automated:overPU_k}
	% affine sys
	\State $\Pu{\Delta t_k} \gets $ \eqref{eq:Pu_init}
		\label{alg:automated:Pu_init}
	\State $\Pu{t_{k+1}} \gets \Pu{t_k} \oplus e^{At_k} \Pu{\Delta t_k}$
		\label{alg:automated:Pu_k}
	\State $\HPtp{t_{k+1}} \gets e^{A t_{k+1}} \initset{} + \Pu{t_{k+1}}$
		\label{alg:automated:HPtp_k}
	\State $\C{} \hspace{-1.5pt} \gets \hspace{-1.5pt} \Fx{k} \overHPtp{t_k} \hspace{-1pt} \oplus \hspace{-1pt} \Fu{k} \uTrans{}$
		\Comment{see \eqref{eq:Fx}, \hspace{-2pt}\eqref{eq:Fu}}
		\label{alg:automated:tayl_C}
	\State $\overHPti{\tau_k} \gets \linCombOp{\HPtp{t_k}, \HPtp{t_{k+1}}} \oplus \C{}$
		\label{alg:automated:overHPti_k}
	\State $\eacc{k+1},\ered{k+1} \gets$ \eqref{eq:eacc_eaccstep}, \eqref{eq:ered}
		\label{alg:automated:eacc_ered}
	% combination
	\State $\overRti{\tau_k} \gets \overHPti{\tau_k} \oplus \overPU{t_{k+1}}$
		\label{alg:automated:overRti}
	% increment counter
	\State $k \gets k+1$
		\label{alg:automated:incrk}
\EndWhile \label{alg:automated:looptend}
\State \Return $\overRti{[0,\tFinal{}]} \gets \bigcup_{j=0}^{k-1} \overRti{\tau_j}$
	\label{alg:automated:fullR}
\end{algorithmic}
\end{algorithm}

The resulting automated tuning algorithm is shown in \algref{alg:automated}:
In the repeat-until loop (\linesref{alg:automated:loopstart}{alg:automated:errcheck}), we first decrease the time step size $\Delta t_k$ (\lineref{alg:automated:shrinkDeltat}) and tune the truncation order $\tayl{k}$ (\linesref{alg:automated:tayl}{alg:automated:tayl_end}) until the respective error bounds $\eaccadmstep{k}$ and $\enonaccadm{k}$ for the accumulating and non-accumulating errors are safisfied.
After this loop, we compute the particular solution due to the input set $\inputsetzero{}$ and tune the zonotope order $\zonorder{k}$ (\linesref{alg:automated:zonorder}{alg:automated:zonorder_end}) yielding the reduction error $\eredstep{k}$.
Afterwards, we compute the solution to the affine dynamics (\linesref{alg:automated:Pu_k}{alg:automated:overHPti_k}) and finally obtain the reachable set of the current time interval (\lineref{alg:automated:overRti}).

The runtime complexity of \algref{alg:automated} is $\bigO{n^3}$ as for the base algorithm, \algref{alg:standard}.
For an initial set $\initset{}$ with zonotope order $\rho_X$ and an input set with zonotope order $\rho_U$, the space complexity for the $k$-th set $\overRti{\tau_k}$ is bounded by $\bigO{n^2(\rho_X + k\rho_U)}$ and the space complexity for \algref{alg:automated} then follows by summing over all individual steps.

\subsection{Proof of Convergence}
\label{ssec:proof}

While \algref{alg:automated} guarantees to return a reachable set $\overRti{[0,\tFinal{}]}$ satisfying the error bound $\emax{}$ by construction, it remains to show that the algorithm terminates in finite time.
To respect the linearly increasing bound for the accumulating error, we have to show that this error decreases faster than linearly with the time step size $\Delta t_k$;
thus, by successively halving the time step size, we will always find a time step size so that the error bound is satisfied.
Using \lmmsref{lmm:E_lim}{lmm:eUstep_lim} from Appendix~B, we now formulate our main theorem:
\begin{theorem}[Convergence] \label{thm:automated}
\algref{alg:automated} terminates in finite time for arbitrary error bounds $\emax{} > 0$.
\end{theorem}
\begin{proof}
By \lmmref{lmm:eaccUstep_lim}, the additional accumulating error $\eaccstep{k}$ decreases quadratically with $\Delta t_k$.
Thus, we are guaranteed to find a time step size that satisfies the linearly decreasing bound $\eaccadmstep{k}$ by successively halving $\Delta t_k$.
The non-accumulating error $\enonacc{k}$ decreases at least linearly with $\Delta t_k$ according to \lmmsref{lmm:ehom_lim}{lmm:eUstep_lim}.
Since the error bound $\enonaccadm{k}$ approaches a constant value greater than $0$ for $\Delta t_k \to 0$, we are therefore always able to safisfy $\enonaccadm{k}$ by reducing the time step size.
The additional reduction error $\eredstep{k}$ can be set to 0 by simply omitting the reduction, which trivially satisfies any bound $\eredadmstep{k}$.
\end{proof}

\noindent Our adaptive algorithm \algref{alg:automated} must be based on a wrapping-free reachability algorithm to guarantee convergence as successive propagation with $e^{A \Delta t_k}$ would eliminate the required faster-than-linear decrease of the accumulating error.

\subsection{Improved Tuning Methods}
\label{ssec:tuningmethods}

While \algref{alg:automated} is guaranteed to converge, there is still room for improvement regarding the computation time.
Hence, we present enhanced methods for adapting the time step size $\Delta t$ and the choice of $\zeta$ in \eqref{eq:eredadm} determining the amount of error that is allocated for reduction.

% approximation functions
\subsubsection{Time Step Size}
\label{par:deltat}

Ideally, the chosen time step size $\Delta t_k$ fulfills the resulting error bounds as tightly as possible.
To this end, we replace the naive adaptation of $\Delta t_k$ in \lineref{alg:automated:shrinkDeltat} of \algref{alg:automated} by regression:
We use the previously obtained error values as data points to define linear and quadratic approximation functions modeling the behavior of the error over $\Delta t_k$, depending on the asymptotic behavior of the respective errors according to \lmmsref{lmm:E_lim}{lmm:eUstep_lim} from Appendix~B.
We then compute an estimate of the time step size required to satisfy the error bounds based on the approximation functions.
This estimate is then refined until the error bounds are satisfied.

\subsubsection{Reduction Error Allocation}
\label{par:zeta}

% finding near-optimal \zeta
The second major improvement is to pre-compute a near-optimal value for the parameter $\zeta$ in \eqref{eq:eredadm} using a heuristic that aims to minimize the zonotope order of the resulting reachable sets.
%since a lower representation size accelerates the computation.
Our heuristic is based on the following observation:
For increasing values of $\zeta$, more margin is allocated to the reduction error and less margin to the accumulating and non-accumulating errors.
Thus, the total number of steps increases because the algorithm has to select smaller time step sizes, yielding a higher zonotope order.
At the same time, the zonotope order can be lowered more due to the larger reduction error margin.
We now want to determine the optimal value of $\zeta$ balancing these two effects.

We first estimate the zonotope order of $\overPU{\tFinal{}}$ using the number of time steps if reduction is completely omitted.
Let us denote the total number of steps for \algref{alg:automated} without reduction ($\zeta = 0$) by $\laststepzero{}$, and the zonotope order of the input set $\inputsetzero{}$ by $\zonorderU{}$.
In each step, the set $e^{At_k} \overPUzero{\Delta t_k} = e^{At_k} \Delta t_k \, \inputsetzero{}$ is added to $\overPU{t_k}$, which iteratively increases the zonotope of $\overPU{t_k}$ by $\zonorderU{}$.
Due to the linear decrease of the total error (see \lmmsref{lmm:eaccUstep_lim}{lmm:eUstep_lim} in Appendix~B), using the value $\zeta = 0.5$ at most doubles the number of steps compared to $\zeta = 0$.
%As shown in \lmmsref{lmm:eaccUstep_lim}{lmm:eUstep_lim}, we have at worst linear decrease over $\Delta t_k$ in all non-accumulating or accumulating errors.
%Consequently, halving the margin for these error groups, i.e., setting $\zeta = 0.5$, entails at most double the number of steps compared to $\zeta = 0$, which also doubles the resulting zonotope order as we omit the potential for reduction for now.
%We generalize this to arbitrary values $\zeta \in [0,1)$ to estimate an upper bound for the zonotope order of $\overPU{\tFinal{}}$:
Therefore, the zonotope order of $\overPU{\tFinal{}} = \overPUzero{\tFinal{}} \oplus \overPUrest{\tFinal{}}$ can be estimated as
\begin{equation} \label{eq:rhoplus}
	\zonorder{}^+(\zeta) = \frac{1}{1 - \zeta} \laststepzero{} \zonorderU{} + 1 ,
\end{equation}
if no order reduction takes place, where the summands represent the orders of $\overPUzero{\tFinal{}}$ and $\overPUrest{\tFinal{}}$, respectively.

Next, we estimate the zonotope order of $\overPU{\tFinal{}}$ for a non-zero value $\zeta > 0$ yielding a non-zero error margin $\eredadm{t} > 0$ that is used to reduce the order of $\overPU{\tFinal{}}$.
Using a fixed time step size $\Delta t$ yielding an integer number $\laststep{} = \frac{\tFinal{}}{\Delta t}$ of time steps, we compute the sequence
\begin{equation*}
	\forall j \in \{1,...,\laststep{}+1\}:
	\eredapx{j} = \errOp{ e^{A (j-1) \Delta t} \Delta t \, \inputsetzero{} } ,
\end{equation*}
which estimates the maximum reduction error in each time step.
Let us introduce the ordering $\pi$ which permutes $\{1,...,\laststep{}+1\}$ such that $\eredapx{\pi_1} < ... < \eredapx{\pi_{\laststep{}+1}}$.
To mimic the accumulation of the reduction error, we introduce the cumulative sum over all $\eredapx{j}$ ordered by $\pi$:
{\setlength{\abovedisplayskip}{2pt}
\begin{equation*}
	\forall j \in \{1,...,\laststep{}+1\}:
	\sigma_j = \sum_{i=1}^{j+1} \eredapx{\pi_i} .
\end{equation*}}%
The maximum reducible order $\zonorder{}^-(\zeta)$ exploits the reduction error bound $\eredadm{\tFinal{}} = \zeta \emax{}$ as much as possible:
\begin{align}
\begin{split} \label{eq:rhominus}
	&\zonorder{}^-(\zeta) = j^*  \zonorderU{}, \\
	&\text{where} \; j^* = \argmax_{j \in \{1,...,\laststep{}+1\}} \sigma_j \leq \zeta \emax{} .
\end{split}
\end{align}
Finally, we combine the two parts \eqref{eq:rhoplus} and \eqref{eq:rhominus} describing the counteracting influences to obtain the following heuristic:
\begin{equation} \label{eq:zetar}
	\zeta = \argmin_{\zeta \in [0,1)} \zonorder{}^+(\zeta) - \zonorder{}^-(\zeta) .
\end{equation}
Since this is a scalar optimization problem, we use a fine grid of different values for $\zeta \in [0,1)$ to estimate the optimal value.

\vspace{-0.2cm}
\subsection{Extension to Output Sets}
\label{ssec:outputset}

We now show how to extend the proposed algorithm to outputs $y(t)$.
The output set $\Ytp{t}$ can be computed by evaluating the output equation \eqref{eq:output} in a set-based manner:
\begin{equation} \label{eq:Y}
	\Ytp{t} = C \Rtp{t} \oplus W \mathcal{V} + q .
\end{equation}
For the outer-approximation error of the output set, we have to account for the linear transformation with the matrix~$C$:
\begin{proposition} \label{prop:eps_Y}
Consider a linear system of the form \eqref{eq:linsys}-\eqref{eq:output}.
Given the error $\etotalx{k}$ \eqref{eq:eps_Rti} of the reachable set $\overRtp{\tau_k}$, the corresponding output set $\overYtp{\tau_k}$ has an error of
\begin{equation} \label{eq:eps_R2Y}
	\dH{\Ytp{\tau_k},\overYtp{\tau_k}} \leq \etotaly{k} := \etotalx{k} \norm{C}_2 .
	% \etotalx{k} \norm{\sqrt{C C^\top}}_2 \, .
\end{equation}
\end{proposition}
\begin{proof}
For the error in the state $x(t)$, we have
\begin{equation*}
	\dH{\Rtp{\tau_k},\overRtp{\tau_k}} \leq \etotalx{k} \overset{\eqref{eq:dHball}}{\Rightarrow} \overRtp{\tau_k} \subseteq \Rtp{\tau_k} \oplus \B{\varepsilon} ,
\end{equation*}
where the hyperball $\B{\varepsilon}$ has radius $\varepsilon = \etotalx{k}$.
Applying the output equation \eqref{eq:output} to the right-hand side yields
\begin{alignat*}{3}
	&
	&&C \overRtp{\tau_k} \oplus W \mathcal{V} + q		&&\subseteq C \Rtp{\tau_k} \oplus C \B{\varepsilon} \oplus W \mathcal{V} + q \\
	&\overset{\eqref{eq:Y}}{\Leftrightarrow} \;\;
	&&\qquad \qquad \quad \;\; \overYtp{\tau_k}			&&\subseteq \Ytp{\tau_k} \oplus C \B{\varepsilon} .
\end{alignat*}
The error in $\overYtp{\tau_k}$ is therefore given by the radius of the smallest sphere enclosing the set $C \B{\varepsilon}$, i.e.,
\begin{align*}
	\rad(C \B{\varepsilon})
	&= \rad \big(\big\{ C z ~ \big| ~ z^\top z \leq \varepsilon \big\}\big) \\
	&= \max_{\smallnorm{z}_2 \leq \varepsilon} \norm{C z}_2
	= \varepsilon \max_{\smallnorm{z}_2 \leq 1} \norm{C z}_2 = \varepsilon \norm{C}_2 ,
\end{align*}
where $\norm{C}_2$ is the largest singular value of $C$.
\end{proof}

%% file: figures/eUstep.tikz
\begin{tikzpicture}[scale=1]
% grid for help
%\draw[gray,thin,step=0.5] (0,0) grid (8.5,5.0);

% left side: exact
% HPti + exact PU
\begin{scope}[xshift=-0.5cm]
\filldraw[fill=good_gray,draw=good_blue,thick]
	(1.25,0.5) -- (0.75,1.25) -- %% from start set
	(1.25,3.55) -- (2.6,4.15) -- (3.75,2.95) -- %% from end set
	(2,0.75) -- cycle; %% from start set
% HPti
\draw[good_blue]
	(1.25,0.5) -- (0.75,1.25) -- %% from start set
	(1.75,3.25) -- (2.5,3.75) -- (3.25,3) -- %% from end set
	(2,0.75) -- cycle; %% from start set
% HPtp k
\draw[good_blue] (0.75,1.25) -- (1.25,2) -- (2.25,1.5) -- (2,0.75);
% HPtp k+1
\draw[good_blue] (1.75,3.25) -- (1.85,2.45) -- (2.65,2.25) -- (3.25,3);
%\node[align=center] at (2.4,3.2) {$\overHPti{\tau_k}$ \\ $\oplus$ \\ $\overPU{\tau_k}$};
\end{scope}

% right side: over-approximative
\begin{scope}[xshift=4.25cm]
% HPti + overPU
\filldraw[fill=good_gray,draw=good_blue,thick]
	(1.15,0.1) -- (0.3,1.3) -- %% from start set
	(1.25,3.55) -- (2.6,4.15) -- (3.75,2.95) -- %% from end set
	(2.5,0.7) -- cycle; %% from start set
% HPti
\draw[good_blue]
	(1.25,0.5) -- (0.75,1.25) -- %% from start set
	(1.75,3.25) -- (2.5,3.75) -- (3.25,3) -- %% from end set
	(2,0.75) -- cycle; %% from start set
% HPtp k, k+1
\draw[good_blue] (0.75,1.25) -- (1.25,2) -- (2.25,1.5) -- (2,0.75);
\draw[good_blue] (1.75,3.25) -- (1.85,2.45) -- (2.65,2.25) -- (3.25,3);
% exact PU
\draw[good_blue,dashed] (0.75,1.25) -- (1.25,3.55);
\draw[good_blue,dashed] (2,0.75) -- (3.75,2.95);
%\node[align=center] at (2.4,3.2) {$\overHPti{\tau_k}$ \\ $\oplus$ \\ $\overPU{t_{k+1}}$};

\end{scope}

% nodes
\node[align=center] (Htauk) at (3.65,1.95) {$\overHPtp{\tau_k}$};
\draw[postEdge,thin] (1.65,2.1) -- (Htauk.west);
\draw[postEdge,thin] (5.9,2.2) -- (Htauk.east);

\node[align=center] (Htk) at (3.25,0.75) {$\HPtp{t_k}$};
\draw[postEdge,thin] (1.1,1) -- (Htk.west);
\draw[postEdge,thin] (5.6,1) -- (Htk.east);

\node[align=center] (Htkplus) at (4.2,3.4) {$\HPtp{t_{k+1}}$};
\draw[postEdge,thin] (2.0,3.2) -- (Htkplus.west);
\draw[postEdge,thin] (6.7,3.2) -- (Htkplus.east);

\end{tikzpicture}

%% file: figures/errors.tikz
\begin{tikzpicture}[scale = 1]

% shifts for ticks
\coordinate (xtickshift) at (0,0.1);
\coordinate (ytickshift) at (0.1,0);
\def\factor{0.15};
% time
\coordinate (tk) at (2.5,0);
\coordinate (tkplus1) at (4.1,0);
\coordinate (tend) at (6,0);
% errors
\coordinate (emax) at (0,4);
\coordinate (ered) at (0,1.8);
\coordinate (eredadmtk) at (0,0.75); %tk/tend * ered
\coordinate (eredadmtkplus1) at (0,1.23); %tkplus1/tend * ered
\coordinate (eredtk) at ($0.7*(eredadmtk)$);
\coordinate (eaccadmtk) at (0,0.9167); %tk/tend * (emax-ered)
\coordinate (eaccadmtkplus1) at (0,1.5033); %tkplus1/tend * (emax-ered)
\coordinate (eacctk) at ($0.7*(eaccadmtk)$);
\def\eaccstepofeaccadmstep{0.8};

% grid to define coordinates
%\draw[gray,thin,step=0.5] (-1,-1) grid ($(tend) + (emax) + (1.5,1)$);

% axes
\draw[axes] (-0.25,0) node[left] {$0$} -- ($(tend) + (1,0)$) node[right] {$t$};
\draw[axes] (0,-0.25) node[below] {$0$} -- ($(emax) + (0,0.5)$) node[above] {$\varepsilon$};
\draw[dotted] ($(emax)$) -- ($(tend) + (emax)$);

% ticks
\draw[thin] ($(tk)+(xtickshift)$) -- ($(tk)-(xtickshift)$) node[below] {$t_k$};
\draw[thin] ($(tkplus1)+(xtickshift)$) -- ($(tkplus1)-(xtickshift)$) node[below,xshift=0.16cm] {$t_k + \Delta t_k$};
\draw[thin] ($(tend)+(xtickshift)$) -- ($(tend)-(xtickshift)$) node[below] {$\tFinal{}$};
\draw[thin] ($(emax)+(ytickshift)$) -- ($(emax)-(ytickshift)$) node[left] {$\emax{}$};

% er:
\begin{scope}[thick,draw=good_blue,text=good_blue]
\draw[thin] (0,0) -- ($(tend) + (ered)$);
\node at ($0.3*(tend) + 0.45*(ered)$) {$\eredadm{t}$};
% current error
\draw (tk) -- ($(tk) + (eredtk)$) node[midway,right] {$\ered{k}$};
\draw ($(tk) + (eredtk) - (ytickshift)$) -- ($(tk) + (eredtk) + (ytickshift)$);
\draw ($(tk) - (ytickshift)$) -- ($(tk) + (ytickshift)$);
% link
\draw[dotted,thin] ($(tk) + (eredtk)$) -- ($(tkplus1) + (eredtk)$);
% step
\draw ($(tkplus1) + (eredtk)$) -- ($(tkplus1) + (eredadmtkplus1)$)
	node[midway,right] {$\eredadmstep{k}$};
\draw ($(tkplus1) + (eredadmtkplus1) - (ytickshift)$) -- ($(tkplus1) + (eredadmtkplus1) + (ytickshift)$);
\draw ($(tkplus1) + (eredtk) - (ytickshift)$) -- ($(tkplus1) + (eredtk) + (ytickshift)$);
% max margin
\draw (tend) -- ($(tend) + (ered) - \factor*(xtickshift)$) node[right,midway] {$\eredadm{\tFinal{}}$};
\draw ($(tend) - (ytickshift)$) -- ($(tend) + (ytickshift)$);
\draw ($(tend) + (ered) - (ytickshift) - \factor*(xtickshift)$)
	-- ($(tend) + (ered) + (ytickshift) - \factor*(xtickshift)$);
% brace version	
%\draw [decorate,decoration={brace,amplitude=10pt}]
%	($(tend) + (ered) + (ytickshift)$) -- ($(tend) + (ytickshift)$)
%	node [midway,xshift=1cm] {$\eredadm{\tFinal{}}$};
\end{scope}

% ea:
\begin{scope}[thick,draw=good_red,text=good_red]
\draw[thin] (0,0) -- ($(tend) + (emax)$);
\node at ($0.3*(tend) + 0.4*(emax)$) {$\eaccadm{t}$};
% current error
\draw ($(tk) + (eredadmtk)$) -- ($(tk) + (eredadmtk) + (eacctk)$)
	node[midway,right,yshift=0.15cm] {$\eacc{k}$};
\draw ($(tk) + (eredadmtk) + (eacctk) - (ytickshift)$) -- ($(tk) + (eredadmtk) + (eacctk) + (ytickshift)$);
\draw ($(tk) + (eredadmtk) - (ytickshift)$) -- ($(tk) + (eredadmtk) + (ytickshift)$);
% link
\draw[dotted,thin] ($(tk) + (eredadmtk) + (eacctk)$) -- ($(tkplus1) + (eredadmtkplus1) + (eacctk)$);
% step adm
\draw ($(tkplus1) + (eredadmtkplus1) + (eacctk) + 0.5*(ytickshift)$)
	-- ($(tkplus1) + (eredadmtkplus1) + (eaccadmtkplus1) + 0.5*(ytickshift)$)
	node[right,midway] {$\eaccadmstep{k}$};
\draw ($(tkplus1) + (eredadmtkplus1) + (eaccadmtkplus1) - 0.5*(ytickshift)$)
	-- ($(tkplus1) + (eredadmtkplus1) + (eaccadmtkplus1) + 1.5*(ytickshift)$);
\draw ($(tkplus1) + (eredadmtkplus1) + (eacctk) - 0.5*(ytickshift)$)
	-- ($(tkplus1) + (eredadmtkplus1) + (eacctk) + 1.5*(ytickshift)$);

% max margin (at tFinal)
\draw ($(tend) + (ered) + \factor*(xtickshift)$) -- ($(tend) + (emax)$)
	node[right,midway] {$\eaccadm{\tFinal{}}$};
\draw ($(tend) + (emax) - (ytickshift)$) -- ($(tend) + (emax) + (ytickshift)$);
\draw ($(tend) + (ered) - (ytickshift) + \factor*(xtickshift)$)
	-- ($(tend) + (ered) + (ytickshift) + \factor*(xtickshift)$);
\end{scope}

% en: black
\begin{scope}[thick]
\draw ($(tkplus1) + (eredadmtkplus1) + (eacctk)$) -- ($(tkplus1) + (emax)$)
	node[midway,left] {$\enonaccadm{k}$};
\draw ($(tkplus1) + (emax) - (ytickshift)$) -- ($(tkplus1) + (emax) + (ytickshift)$);
\draw ($(tkplus1) + (eredadmtkplus1) + (eacctk) - (ytickshift)$)
	-- ($(tkplus1) + (eredadmtkplus1) + (eacctk) + (ytickshift)$);
\end{scope}

\end{tikzpicture}

%% file: subtex/post.tex
\section{Inner-Approximations}
\label{sec:innerApprox}

As shown in \secref{sec:tuning}, the outer-approximation $\overRtp{t}$ computed by \algref{alg:automated} has a Hausdorff distance of at most $\emax{}$ to the exact reachable set $\Rtp{t}$.
Consequently, an inner-approximation $\underRtp{t} \subseteq \Rtp{t}$ can be computed by the Minkowski difference $\underRtp{t} = \overRtp{t} \ominus \B{\varepsilon}$ of the outer-approximation and the hyperball $\B{\varepsilon}$ with radius $\varepsilon = \emax{}$.
Note that one can also replace $\emax{}$ by the computed error from \algref{alg:automated} to obtain a tighter inner-approximation.
Unfortunately, there exists no closed formula for the Minkowski difference of a zonotope and a hyperball.
Therefore, we first enclose $\B{\varepsilon}$ with a polytope $\P{} \supseteq \B{\varepsilon}$ since the Minkowski difference of a zonotope and a polytope can be computed efficiently if the resulting set is represented by a constrained zonotope:
\begin{proposition}[Minkowski difference] \label{prop:minkDiff}
Given a zonotope $\Z{} = \zono{c,G} \subset \R{n}$ and a polytope $\P{} = \polyV{[v_1~\dots~v_{\polyVert}]} \subset \R{n}$, their Minkowski difference can be represented by the constrained zonotope
{\setlength{\belowdisplayskip}{2pt}
\begin{equation*}
	\Z{} \ominus \P{} = \conZono{c - v_1, [G~\matzeros{}],A,b},
\end{equation*}}%
where
{\setlength{\abovedisplayskip}{2pt}
\begin{equation*}
\begin{split}
& A = \begin{bmatrix} G & -G & \dots & \matzeros{} \\ \vdots & \vdots & \ddots & \vdots \\ G & \matzeros{} & \dots & -G \end{bmatrix},
~~ b = \begin{bmatrix} v_1 - v_2 \\ \vdots \\ v_1 - v_\polyVert \end{bmatrix}.
\end{split}
\end{equation*}}%
\end{proposition}
\begin{proof}
According to \cite[Lemma~1]{Althoff2015d}, the Minkowski difference with a polytope as minuend can be computed as
\begin{equation*}
	\begin{split}
		\Z{} \ominus \P{} &= \hspace{-10pt} \bigcap_{i \in \{1,\dots,\polyVert \}} \hspace{-10pt} (\Z{} - v_i) = (\Z{} - v_1) \cap \dotsc \cap (\Z{} - v_\polyVert).
	\end{split}
\end{equation*}
Using the equation for the intersection of constrained zonotopes in \cite[Eq.~(13)]{Scott2016}, we obtain for the first intersection
\begin{align*}
	\begin{split}
		&(\Z{} - v_1) \cap (\Z{} - v_2) \\
		&\hspace{16pt} = \conZono{ c - v_1,G,[~],[~] } \cap \conZono{ c - v_2,G,[~],[~] } \\
		&\overset{\text{\cite[Eq.~(13)]{Scott2016}}}{=} \conZono{ c - v_1,G,[G~-\!G], v_1 - v_2 }.
	\end{split}
\end{align*}
Repeated application of \cite[Eq.~(13)]{Scott2016} yields the claim.
\end{proof}

For general polytopes the number of vertices increases exponentially with the system dimension.
To keep the computational complexity small, we enclose the hyperball $\B{\varepsilon}$ by a cross-polytope $\polyV{ \varepsilon \, \sqrt{n} \, [-I_n~ I_n]} \supseteq \B{\varepsilon}$, which is a special type of polytope with only $2n$ vertices.
Since the Hausdorff distance between the hyperball and the enclosing cross-polytope is $(\sqrt{n}-1) \varepsilon$, we scale the error bound $\emax{}$ by the factor $1/\sqrt{n}$ before executing \algref{alg:automated} in order to obtain an inner-approximation with a maximum Hausdorff distance of $\emax{}$ to the exact reachable set.
%
%The Hausdorff distance between the hyperball and the enclosing cross-polytope is $(\sqrt{n}-1) \varepsilon$.
%To ensure that the inner-approximation has a maximum Hausdorff distance of $\emax{}$ to the exact reachable set, we therefore have to scale the allowed error by the factor $1/\sqrt{n}$ before executing \algref{alg:automated}. 

% -----------------------------------------------------------------------

\section{Automated Verification}
\label{sec:verification}

One core application of reachability analysis is the verification of safety specifications.
Based on our automated parameter tuning approach, we introduce a fully automated verification algorithm for linear systems, which iteratively refines the tightness of the reachable set inner-approximation and outer-approximation until a given specification can be verified or falsified. 
We consider specifications of the form
\begin{equation*}
	\forall t \in [0,\tFinal{}]:~ \bigg( \bigwedge_{i=1}^{\nrSafeSets{}} \mathcal{R}(t) \subseteq \safeSet{i} \bigg) \wedge
	\bigg( \bigwedge_{i=1}^{\nrUnsafeSets{}} \mathcal{R}(t) \cap \unsafeSet{i} = \emptyset \bigg)
\end{equation*}
defined by a list of safe sets $\{\safeSet{1},\dots,\safeSet{\nrSafeSets{}}\} \subset \R{n}$ and a list of unsafe sets $\{\unsafeSet{1},\dots,\unsafeSet{\nrUnsafeSets{}}\} \subset \R{n}$, both specified as polytopes in halfspace representation.
While we omit the dependence on time here for simplicity, the extension to time-varying safe sets and unsafe sets is straightforward.
To check if the reachable set satisfies the specification, we need to perform containment and intersection checks on zonotopes and constrained zonotopes:

\begin{proposition}[Containment check] \label{prop:containment}
Given a polytope $\P{} = \poly{C,d} \subset \R{n}$ and a constrained zonotope $\CZ{} = \conZono{c,G,A,b} \subset \R{n}$, we have
{\setlength{\belowdisplayskip}{2pt}
\begin{equation} \label{eq:containmentConZono}
	\CZ{} \subseteq \P{} \;
	\Leftrightarrow \;
	\underbrace{\max \big(\distance{}_1,\dots,\distance{}_\polyCons{} \big) }_{\distance{}} \leq 0,
\end{equation}}%
where each linear program
\begin{align*}
	\forall i \in \{1,\dots,\polyCons{} \}: ~~\distance{}_i = &\max_{\alpha \in \R{\gens}} ~ C_{(i,\cdot)} c + C_{(i,\cdot)} G \alpha -d_{(i)} \\ 		
		&\; \text{s.t.} ~~ \alpha \in [-\vecones,\vecones],~ A \, \alpha = b.
\end{align*} 
computes the distance to a single polytope halfspace.
\end{proposition}
\begin{proof}
In general, a set $\S{} \subset \R{n}$ is contained in a polytope if it is contained in all polytope halfspaces.
The linear program above evaluates the support function (see \cite[Def.~1]{LeGuernic2010}) of $\S{}$ along the normal vector of each halfspace, which has to be smaller or equal to the corresponding offset to prove containment \cite[Corollary~13.1.1]{Rockafellar1972}.
\end{proof}
\noindent For a zonotopic in-body $\Z{} = \zono{c,G} \subset \R{n}$, there is a closed-form solution (\cite[Corollary~13.1.1]{Rockafellar1972} with \cite[Prop.~1]{LeGuernic2010}):
{\setlength{\abovedisplayskip}{3pt}
\begin{equation} \label{eq:containmentZono}
	\Z{} \subseteq \P{} \;
	\Leftrightarrow \;
	\underbrace{\max \bigg(Cc-d + \sum_{i=1}^\gens{} |C G_{(\cdot,i)}| \bigg)}_{\distance{}} \leq 0 .
\end{equation}}%
Next, we consider intersection checks:
\begin{proposition}[Intersection check] \label{prop:isIntersecting}
A polytope $\P{} = \poly{C,d} \subset \R{n}$ and a constrained zonotope $\CZ{} = \conZono{c,G,A,b} \subset \R{n}$
intersect if $\distance{} \leq 0$ computed by the linear program
%\begin{align*}
%	&\distance{} = \min_{\substack{x \in \R{n} \\ \alpha \in \R{\gens{}} \\ \delta \in \R{}}} \delta ~~
%	\text{s.t.} ~~ \forall i \in \{1,...,a\}: C_{(i,\cdot)} x - d_{(i)} \leq \delta , \\
%	&\hspace{77pt} x = c + G \alpha, ~ A \alpha = b , ~ \alpha \in [-\vecones,\vecones] .
%\end{align*}
\begin{align*}
	&\distance{} = \min_{x \in \R{n}, \, \alpha \in \R{\gens{}}, \, \delta \in \R{}} \delta \\
	&\hspace{41pt} \text{s.t.} ~~ \forall i \in \{1,...,a\}: C_{(i,\cdot)} x - d_{(i)} \leq \delta , \\
	&\hspace{41pt} x = c + G \alpha, ~ A \alpha = b , ~ \alpha \in [-\vecones,\vecones] .
\end{align*}
\end{proposition}
\begin{proof}
If $\forall i \in \{1,...,\polyCons{}\}: C_{(i,\cdot)}x-d_{(i)} \leq \delta \leq 0$, then there exists a point $x \in \CZ{}$ that is also contained in $\P{}$.
\end{proof}

\begin{algorithm}[!tb]
\caption{Automated verification} \label{alg:verify}
\textbf{Require:} Linear system $\dot{x} = Ax + Bu + p$,
	initial set~$\initset{} = \zono{c_x,G_x}$, input set~$\inputset{} = \zono{c_u,G_u}$,
	time horizon $\tFinal{}$,
	specification defined by a list of safe sets $\safeSet{1},\dots,\safeSet{\nrSafeSets{}}$
	and a list of unsafe sets $\unsafeSet{1},\dots,\unsafeSet{\nrUnsafeSets{}}$
	
\textbf{Ensure:} Specification satisfied (true) or violated (false)

\setstretch{1.1}
\begin{algorithmic}[1]
	\State $\emax{} \gets \text{estimated from simulations}$ \label{alg:verifyInitGuess}
	\Repeat \label{alg:verifyBeginRepeat}
		\State $\overRtp{t} \gets $ comp. with \algref{alg:automated} using error bound $\emax{}$ \label{alg:verifyBeginReach}
		\State $\underRtp{t} \gets \overRtp{t} \ominus \B{\varepsilon}$
			\label{alg:verifyEndReach} \Comment{see \secref{sec:innerApprox}}
		\State $\widehat{\distance}_G,\widecheck{\distance}_G \gets -\infty$, $\widehat{\distance}_F, \widecheck{\distance}_F \gets \infty$
		\For{$j \gets 1$ to $\nrSafeSets{}$} \label{alg:verifyBeginCheck}
			\State $\distance{} \gets \text{distance from } \overRtp{t} \subseteq \safeSet{j}$
				\Comment{see \eqref{eq:containmentZono}}
			\State $\widehat{\distance{}}_G \gets \max(\widehat{\distance{}}_G,\distance{})$
			\State $\distance{} \gets \text{distance from } \underRtp{t} \subseteq \safeSet{j}$
				\Comment{see \eqref{eq:containmentConZono}}
				\label{alg:verify:containmentConZono}
			\State $\widecheck{\distance{}}_G \gets \max(\widecheck{\distance{}}_G,\distance{})$
		\EndFor 
		\For{$j \gets 1$ to $\nrUnsafeSets{}$}
			\State $\distance{} \gets \text{distance from } \overRtp{t} \cap \unsafeSet{j} = \emptyset$
				\Comment{see \propref{prop:isIntersecting}}
				\label{alg:verify:isIntersectingZono}
			\State $\widehat{\distance{}}_F \gets \min(\widehat{\distance{}}_F,\distance{})$
			\State $\distance{} \gets \text{distance from } \underRtp{t} \cap \unsafeSet{j} = \emptyset$
				\Comment{see \propref{prop:isIntersecting}}
				\label{alg:verify:isIntersectingConZono}
			\State $\widecheck{\distance{}}_F \gets \min(\widecheck{\distance{}}_F,\distance{})$
		\EndFor	\label{alg:verifyEndCheck}
		\State $\distance{} \gets \min (-\widecheck{\distance{}}_G,\widecheck{\distance{}}_F)$ \label{alg:verifyBeginMinDist}
		\If{$\widehat{\distance{}}_G \geq 0$}
			\State $\distance{} \gets \min (\distance{},\widehat{\distance{}}_G)$
		\EndIf
		\If{$\widehat{\distance{}}_F \leq 0$}
			\State $\distance{} \gets \min (\distance{},-\widehat{\distance{}}_F)$
		\EndIf \label{alg:verifyEndMinDist}
		\State $\emax{} \gets \max\big(0.1 \, \emax{}, \min(\distance{},0.9 \, \emax{})\big)$ \label{alg:verifyUpdate}
	\Until{$\big( (\widehat{\distance{}}_G \leq 0) \wedge (\widehat{\distance{}}_F > 0) \big) \vee (\widecheck{\distance{}}_G > 0) \vee (\widecheck{\distance{}}_F \leq 0)$} \label{alg:verifyEndRepeat}
	\State \Return $(\widehat{\distance{}}_G \leq 0) \wedge (\widehat{\distance{}}_F > 0)$
\end{algorithmic}
\end{algorithm}

Since a zonotope is just a special case of a constrained zonotope, \propref{prop:isIntersecting} can also be used to check if a zonotope intersects a polytope.
For both \propref{prop:containment} and \propref{prop:isIntersecting}, $\distance{}$ is a good estimate for the Hausdorff distance between the sets if polytopes with normalized halfspace normal vectors are used.
We utilize this in our automated verification algorithm to estimate the accuracy that is required to verify or falsify the specification.
%If the outer-approximation of the reachable set intersects an unsafe set, we furthermore use the diameter of the intersecting area to estimate the required accuracy.
%Given a zonotope $\Z{} = \zono{c,G} \subset \R{n}$ and a polytope $\P{} = \poly{C,d} \subset \R{n}$ that intersect, this diameter can be estimated by
%
%\begin{equation} \label{eq:distIntersectionUnsafe}
%	\distance{} = \sqrt{ \big(\overline{x} - \underline{x}\big)^\top \big(\overline{x} - \underline{x}\big)}, ~~~~ [\underline{x},\overline{x}] = \boxOp{ \Z{} \cap \P{} }
%\end{equation}
%
%where the intersection $\Z{} \cap \P{}$ is computed using constrained zonotopes:
%
%\begin{equation*}
%\begin{split}
%	& \Z{} \cap \P{} = \conZono{c,G,[~],[~]} \cap \poly{C,d} = \conZono{c,G,[~],[~]} \, \cap \\ 
%	& ~ \big \{ x \in \R{n} \, \big | \, C_{(1,\cdot)} x \leq d_{(1)} \big \} \cap \dotsc \cap \big \{ x \in \R{n} \, \big | \, C_{(\polyCons,\cdot)} x \leq d_{(\polyCons)} \big \}.
%\end{split}
%\end{equation*}
%
%The intersections of the constrained zonotope and the halfspaces $\big \{ x \in \R{n} \, \big | \, C_{(i,\cdot)} x \leq d_{(i)} \big \}$ where $i \in \{1,...,\polyCons{}\}$ are then computed according to \cite[Sec.~3.2]{Raghuraman2020}, and the box enclosure in \eqref{eq:distIntersectionUnsafe} is computed using linear programming \cite[Prop.~1]{Rego2018}.

The overall verification algorithm is summarized in \algref{alg:verify}:
We first obtain an initial guess for the error bound $\emax{}$ in \lineref{alg:verifyInitGuess} by simulating trajectories for a finite set of points from $\initset{}$.
The repeat-until loop (\linesref{alg:verifyBeginRepeat}{alg:verifyEndRepeat}) then refines the inner- and outer-approximations of the reachable set by iteratively decreasing the error bound $\emax{}$ until the specifications can be verified or falsified.
In particular, we first compute the outer- and inner-approximation (\linesref{alg:verifyBeginReach}{alg:verifyEndReach}).
Next, we perform the containment and intersection checks with the safe and unsafe sets (\linesref{alg:verifyBeginCheck}{alg:verifyEndCheck}) and store the corresponding distances $\widehat{\distance}_G,\widecheck{\distance}_G,\widehat{\distance}_F,\widecheck{\distance}_F$.
Using these distances, we determine the minimum distance (\linesref{alg:verifyBeginMinDist}{alg:verifyEndMinDist}), which is then used to update the error bound $\emax{}$ (\lineref{alg:verifyUpdate}).
Since $\distance{}$ is only an estimate, we also restrict the updated error bound to the interval $[0.1\, \emax{}, 0.9\,\emax{}]$ to guarantee convergence and avoid values that are too small.
Finally, the specifications are satisfied if the outer-approximation of the reachable set is contained in all safe sets ($\widehat{\distance{}}_G \leq 0$) and does not intersect any unsafe sets ($\widehat{\distance{}}_F > 0$).
On the other hand, the specifications are falsified if the inner-approximation of the reachable set is not contained in all safe sets ($\widecheck{\distance{}}_G > 0$) or intersects an unsafe set ($\widecheck{\distance{}}_F \leq 0$).
Improvements for \algref{alg:verify} which we omitted here for simplicity include using the computed error from \algref{alg:automated} instead of the error bound $\emax{}$, omitting re-computation as well as containment and intersection checks for time intervals that are already verified, and only computing inner-approximations if the corresponding outer-approximation is not yet verified.  

The space complexity of \algref{alg:verify} is dominated by the inner-approximation $\underRti{\tau_k}$, which is $\bigO{n^4}$ following \propref{prop:minkDiff}.
Assuming a conservative bound of $\bigO{p^{3.5}}$ for a linear program with $p$ variables according to \cite{Karmarkar1984}, the runtime complexity of \algref{alg:verify} is $\bigO{n^{7}}$ since we evaluate linear programs in Props.~\ref{prop:containment}-\ref{prop:isIntersecting} with $\bigO{n^2}$ variables, respectively.

%% file: subtex/numex.tex
\section{Numerical Examples}
\label{sec:numex}

Let us now demonstrate the performance of our adaptive tuning approach and our verification algorithm.
We integrated both algorithms into the MATLAB toolbox CORA \cite{Althoff2015a}, and they will be made publicly available with the 2023 release\footnote{available at \url{https://cora.in.tum.de}}.
All computations are carried out on a 2.59GHz quad-core i7 processor with 32GB memory.

\subsection{Electrical Circuit}
\label{ssec:circuit}

%\begin{figure}[!tb]
%  % \setlength{\belowcaptionskip}{-10pt}
%  \centering
%  \includegraphics[width=0.9\columnwidth]{./figures/electricCircuit}
%  \caption{Circuit diagram for the electric circuit benchmark.}
%  \label{fig:electricCircuit}
%\end{figure}

To showcase the general concept of our approach, we first consider the deliberately simple example of an electric circuit consisting of a resistance $R = 2\si{\ohm}$, a capacitor with capacity $C = 1.5\si{\milli \farad}$, and a coil with inductance $L = 2.5\si{\milli \henry}$:
\begin{equation*}
	\begin{bmatrix} \dot u_C(t) \\ \dot i_L(t)\end{bmatrix}
	= \begin{bmatrix} -\frac{1}{RC} & \frac{1}{C} \\ -\frac{1}{L} & 0 \end{bmatrix} \begin{bmatrix} u_C(t) \\ i_L(t) \end{bmatrix} + \begin{bmatrix} 0 \\ \frac{1}{L} \end{bmatrix} u_I(t),
\end{equation*}
where the state is defined by the voltage at the capacitor $u_C(t)$ and the current at the coil $i_L(t)$.
The initial set is $\initset{} = [1,3]\si{\volt} \times [3,5]\si{\ampere}$, the input voltage to the circuit $u_I(t)$ is uncertain within the set $\inputset{} = [-0.1,0.1]\si{\volt}$, and the time horizon is $\tFinal{} = 2\si{\second}$.
As shown in \figref{fig:reachSetElectricCircuit}, the inner- and outer-approximations computed using \algref{alg:automated} and \secref{sec:innerApprox} converge to the exact reachable set with decreasing error bounds.
The computation times are $0.26\si{\second}$ for $\emax{} = 0.04$, $0.41\si{\second}$ for $\emax{} = 0.02$, and $0.55\si{\second}$ for $\emax{} = 0.01$.

\begin{figure}[!tb]
  \centering
  \includegraphics[width=0.99\columnwidth]{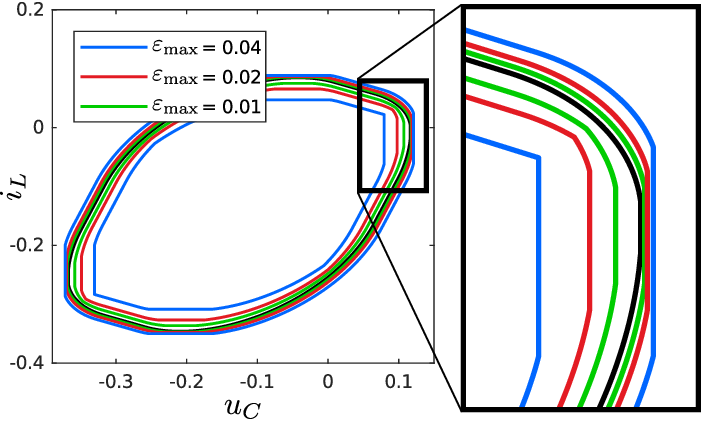}
  \caption{Inner- and outer-approximations of the final reachable set $\Rti{\tFinal{}}$ for the electric circuit using different error bounds $\emax{}$, with the exact reachable set is shown in black.}
  \label{fig:reachSetElectricCircuit}
\end{figure}

\subsection{ARCH Benchmarks}
\label{ssec:ARCH}

Next, we evaluate our verification algorithm on benchmarks from the 2021 ARCH competition \cite{ARCH2021linear}, where state-of-the-art reachability tools compete with one another to solve challenging verification tasks.
We consider all linear continuous-time systems, which are the building benchmark (BLD) describing the movement of an eight-story hospital building, the International Space Station (ISS) benchmark modeling a service module of the ISS, the Heat 3D benchmark (HEAT) representing a spatially discretized version of the heat equation, and the clamped beam benchmark (CB) monitoring oscillations of a beam.
The results in \tabref{tab:benchmarks} demonstrate that our fully automated verification algorithm correctly verifies all safe benchmarks without being significantly slower than state-of-the-art tools that require extensive parameter tuning by experts.
Please note that we compare only to the computation time of other tools achieved by the optimal run with expert-tuned algorithm parameters, disregarding the significant amount of time required for tuning.
Moreover, our algorithm also successfully falsifies the two unsafe benchmarks, where our computation time is slightly worse because the other tools do not explicitly falsify these benchmarks but only test if they cannot be verified, which is considerably easier.

\input{./figures/ARCHtable.tex}

\subsection{Autonomous Car}
\label{sec:autonCar}

Finally, we show that our verification algorithm can handle complex verification tasks featuring time-varying specifications.
To this end, we consider the benchmark proposed in \cite{Kochdumper2021b}, where the task is to verify that a planned reference trajectory $x_{\text{ref}}(t)$ tracked by a feedback controller is robustly safe despite disturbances and measurement errors.
The nonlinear vehicle model in \cite[Eq.~(3)]{Kochdumper2021b} is replaced by a linear point mass model, which yields the closed-loop system
\begin{equation*}
	\begin{bmatrix} \dot x(t) \\ \dot x_{\text{ref}}(t) \end{bmatrix}
	= \begin{bmatrix} A + BK & \hspace{-5pt}-BK \\ \matzeros & \hspace{-5pt}A \end{bmatrix} \begin{bmatrix} x(t) \\ x_{\text{ref}}(t) \end{bmatrix}
	+ \begin{bmatrix} B & \hspace{-5pt}B & \hspace{-5pt}BK \\ B & \hspace{-5pt}\matzeros & \hspace{-5pt}\matzeros \end{bmatrix} u(t)
\end{equation*}
with $A = [\matzeros{} ~ [I_2~\matzeros{}]^\top]$, $B = [\matzeros{}~I_2]^\top$, and feedback matrix $K \in \R{2 \times 4}$. The initial set is $\initset{} = (x_0 + \mathcal{V}) \times x_0$ and the set of uncertain inputs is $\inputset{} = u_{\text{ref}}(t) \times \mathcal{W} \times \mathcal{V}$, where $\mathcal{W} \subset \R{2}$ and $\mathcal{V} \subset \R{4}$ are the sets of disturbances and measurement errors taken from \cite[Sec.~3]{Kochdumper2021b}, and the initial state $x_0 \in \R{4}$ and control inputs for the reference trajectory $u_{\text{ref}}(t) \in \R{2}$ are specific to the considered traffic scenario.
To compute occupied space of the car, we apply affine arithmetic \cite{deFigueiredo2004} to evaluate the nonlinear map in \cite[Eq.~(4)]{Kochdumper2021b}, where we determine the orientation of the car from the direction of the velocity vector.

For verification, we consider the traffic scenario \textit{BEL\_Putte-4\_2\_T-1} from the CommonRoad database\footnote{available at \url{https://commonroad.in.tum.de/scenarios}}.
% commonroad paper: \cite{Althoff2017a}.
The unsafe sets $\unsafeSet{i}$ for the verification task are given by the road boundary and the occupancy space of other traffic participants.
Since the road boundary is non-convex, we use triangulation to represent it as the union of 460 convex polytopes.
The occupancy spaces of other traffic participants over time intervals of length $0.1\si{\second}$ are represented by polytopes, which results in 170 time-varying unsafe sets for the six vehicles in the scenario.
The safe set $\safeSet{i}$ is given by the constraint that the absolute acceleration should stay below $11.5\si{\metre \per \square \second}$, which we inner-approximate by a polytope with 20 halfspaces.
Even for this complex verification task, \algref{alg:verify} only requires $64$\si{\second} and two refinements of the error bound $\emax{}$ to prove that the reference trajectory is robustly safe (see \figref{fig:commonRoad}).

\begin{figure}[!tb]
  \centering
  \includegraphics[width=0.99\columnwidth]{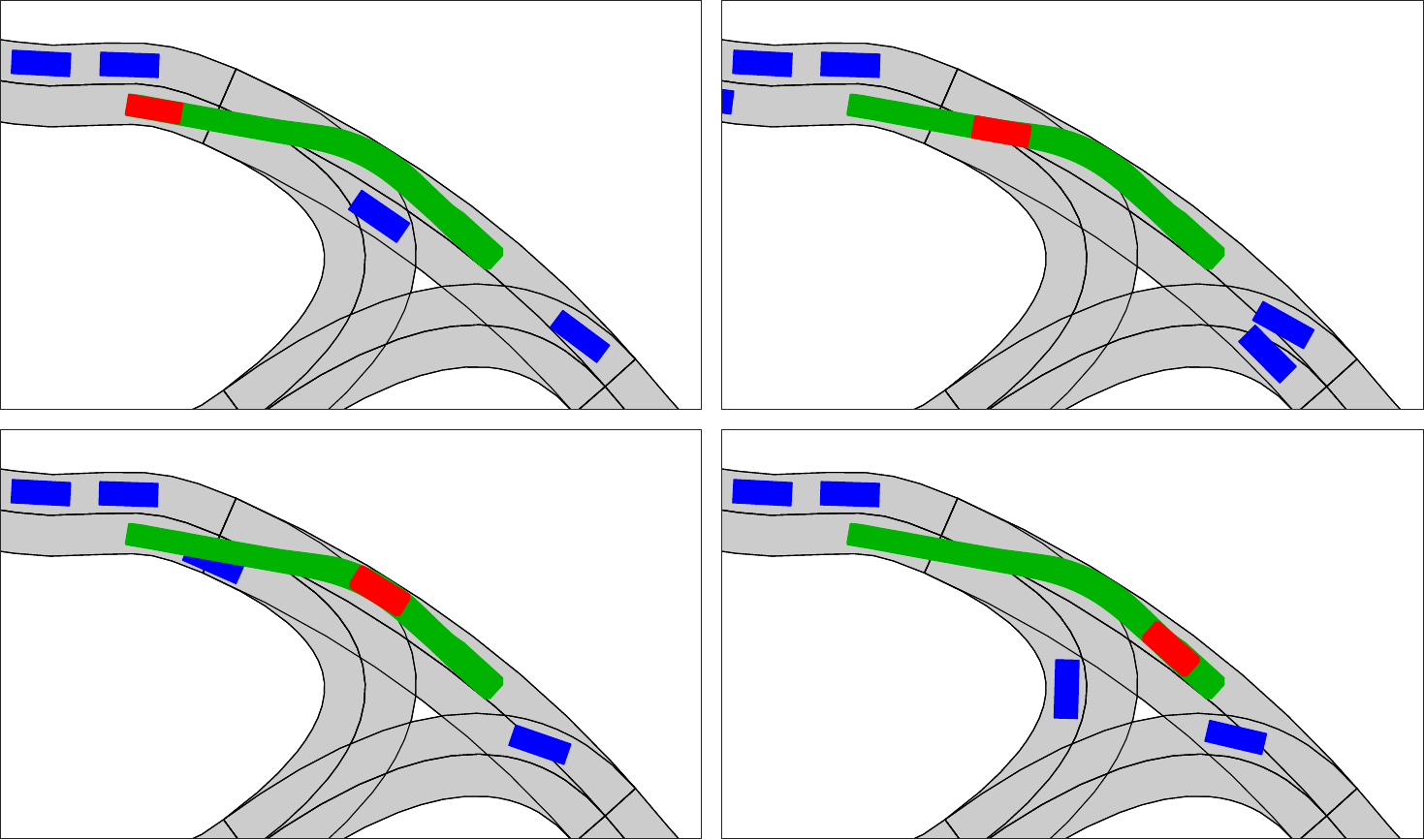}
  \caption{Traffic scenario at times 0\si{\second}, 1\si{\second}, 2\si{\second}, and 3\si{\second}, where the reachable set for the whole time horizon, the reachable set for the current time point, and the other traffic participants are shown.}
  \label{fig:commonRoad}
\end{figure}

%% file: figures/ARCHtable.tex
\begin{table*}[!tbh]
\begin{center}
\caption{Comparison of computation times on the ARCH benchmarks, where $n$ is the system dimension, $m$ is the number of inputs, and $\ell$ is the output dimension.
For our approach we additionally specify the number of refinement iterations of \algref{alg:verify}.
The computation times of the other tools are taken from \cite{ARCH2021linear}.}
\label{tab:benchmarks}
%\vspace{-0.25cm}
%\resizebox{\textwidth}{!}{%
\begin{tabular}{l c c c c  c  c c  c  c c c c}
% Benchmark                			   |       | | verify							 | comparison
% ID | sysdim | input dim | output dim | safe? | | comp. time | #iterations | | tool#1 | tool#2 | tool#3
\toprule
\multicolumn{5}{c}{\textbf{Benchmark}} & \multicolumn{3}{c}{$~~~$Our approach} & & \multicolumn{4}{c}{Time comparison} \\
Identifier & $n$ & $m$ & $\ell$ & Safe? & & Time & Iterations & & CORA & HyDRA & JuliaReach & SpaceEx \\ \midrule
HEAT01 & 125 & 0 & 1 & \cmark & & $2.2\si{\second}$ & 2 & & $2.2\si{\second}$ & $13.2\si{\second}$ & $0.13\si{\second}$ & $4.2\si{\second}$ \\
HEAT02 & 1000 & 0 & 1 & \cmark & & $59\si{\second}$ & 1 & & $9.3\si{\second}$ & $160\si{\second}$ & $32\si{\second}$ & --- \\ \midrule
CBC01 & 201 & 0 & 1 & \cmark & & $28\si{\second}$ & 1 & & $7.1\si{\second}$ & --- & $1.4\si{\second}$ & $312.78\si{\second}$ \\
CBF01 & 200 & 1 & 1 & \cmark & & $144\si{\second}$ & 2 & & $30\si{\second}$ & --- & $12\si{\second}$ & $318.88\si{\second}$ \\ \midrule
BLDC01-BDS01 & 49 & 0 & 1 & \cmark & & $1.7\si{\second}$ & 1 & & $2.9\si{\second}$ & $0.426\si{\second}$ & $0.0096\si{\second}$ & $1.6\si{\second}$ \\
BLDF01-BDS01 & 48 & 1 & 1 & \cmark & & $2.1\si{\second}$ & 1 & & $3.3\si{\second}$ & --- & $0.012\si{\second}$ & $1.8\si{\second}$ \\ \midrule
ISSC01-ISS02 & 273 & 0 & 3 & \cmark & & $4.3\si{\second}$ & 1 & & $1.3\si{\second}$ & --- & $1.4\si{\second}$ & $29\si{\second}$ \\
ISSC01-ISU02 & 273 & 0 & 3 & \xmark & & $10\si{\second}$ & 4 & & $0.072\si{\second}$ & --- & $1.4\si{\second}$ & $29\si{\second}$ \\
ISSF01-ISS01 & 270 & 3 & 3 & \cmark & & $75\si{\second}$ & 2 & & $59\si{\second}$ & --- & $10\si{\second}$ & $49\si{\second}$ \\
ISSF01-ISU01 & 270 & 3 & 3 & \xmark & & $191\si{\second}$ & 3 & & $38\si{\second}$ & --- & $10\si{\second}$ & $48\si{\second}$ \\
\bottomrule
\end{tabular}
%} %end of resizebox
\end{center}
\vspace{-15pt}
\end{table*}

%% file: subtex/discussion.tex
\section{Discussion}
\label{sec:discussion}

Despite the convincing results of our automated verification algorithm in \secref{sec:numex}, there is still potential for improvement:
According to \tabref{tab:benchmarks}, other reachability tools solve high-dimensional benchmarks often faster than our approach since they apply tailored algorithms, such as block-decomposition \cite{Bogomolov2018} and Krylov subspace methods \cite{Althoff2019,Bak2019}.
Therefore, a natural next step is to extend our concept of automated parameter tuning via error analysis to these specialized algorithms to accelerate the verification of high-dimensional systems.
In addition, since many reachability algorithms for nonlinear systems \cite{Althoff2008c,Li2020} are based on reachability analysis for linear systems, another intriguing research direction is the extension to nonlinear systems.
Verification algorithms based on implicit set representations, such as support functions \cite{Wetzlinger2023}, also present an interesting comparison to our proposed method, which computes explicit sets.

Moreover, for systems where some states have no initial uncertainty and are not influenced by uncertain inputs, our proposed algorithm cannot falsify the system since the computed inner-approximation of the reachable set will always be empty.
An example is the autonomous car in \secref{sec:autonCar}, where the states corresponding to the reference trajectory are not subject to any uncertainty.
Fortunately, these cases are easy to detect and one can use classical safety falsification techniques, such as Monte Carlo methods \cite{Abbas2013}, Bayesian optimization \cite{Mathesen2021}, or cross-entropy techniques \cite{Sankaranarayanan2012} instead.
In general, combining safety falsification methods with our algorithm may accelerate the falsification process and provide the user with a concrete counterexample in the form of a falsifying trajectory.

Finally, while our algorithm already supports the general case of specifications defined by time-varying safe sets and unsafe sets, future work could include an extension to temporal logic specifications.
This may be realized by a conversion to reachset temporal logic \cite{Roehm2016b}, a particular type of temporal logic that can be evaluated on reachable sets directly.
Another possibility is to convert temporal logic specifications to an acceptance automaton \cite{Maler2006}, which can then be combined with the linear system via parallel composition \cite{Frehse2018}.

%% file: subtex/conclusion.tex
\section{Conclusion}
\label{sec:conclusion}

In this work, we propose a paradigm shift for reachability analysis of linear systems:
Instead of requiring the user to manually tune algorithm parameters such as the time step size, our approach automatically adapts all parameters such that the computed outer-approximation respects a desired maximum distance to the exact reachable set.
Building on this result, we then extract an inner-approximation of the reachable set directly from the outer-approximation using the Minkowski difference, which finally enables us to design a sound verification algorithm that automatically refines the inner- and outer-approximations until specifications given by time-varying safe and unsafe sets can either be verified or falsified.
An evaluation on benchmarks representing the current limits for state-of-the-art reachability tools demonstrates that our approach is competitive regarding the computation time, even for high-dimensional systems.
Overall, the autonomy of our approach enables non-experts to verify or falsify safety specifications for linear systems in reasonable time.

%% file: subtex/addprops.tex
\renewcommand{\thesectiondis}[2]{\Alph{section}:}
\section{Additional Propositions}
\label{app:addprops}

For the proof of \propref{prop:eps_affine} we require the error induced by outer-approximating the linear combination for zonotopes:
\begin{proposition} \label{prop:ecomb}
Given the homogeneous solution $\HPtp{t_k} = \zono{c_h, G_h}$ with $G_h \in \R{n \times \gens{}_h}$ and the particular solution $\Pu{\Delta t_{k}} = c_p$, the Hausdorff distance between the exact linear combination
\begin{equation*}
	\S{} = \linCombOp{\HPtp{t_k},\HPtp{t_{k+1}}}
\end{equation*}
with $\HPtp{t_{k+1}} = e^{A \Delta t_k} \HPtp{t_{k}} + \Pu{\Delta t_{k}}$ and the corresponding zonotope outer-approximation $\Sover{}$ computed using \eqref{eq:zonoLinComb} is bounded by
\begin{equation*}
	\dH{\S{},\Sover{}} \leq  \sqrt{\gens{}_h} ~ \norm{G_h^{(-)}}_2,
\end{equation*}
where $G_h^{(-)} = (e^{A\Delta t_k} - I_n) G_h$.
\end{proposition}
\begin{proof}
We insert the homogeneous and particular solutions into \eqref{eq:defLinComb} and shift the interval of $\lambda$ from $[0,1]$ to $[-1,1]$, which yields
\begin{align*}
	\S{} &= \bigg\{ \lambda \bigg(c_h + \sum_{i=1}^{\gens{}_h} G_{h(\cdot,i)} \alpha_i \bigg)
		+ (1-\lambda) \bigg( e^{A\Delta t_k} \\
	&\;\; \bigg(c_h + \sum_{i=1}^{\gens{}_h} G_{h(\cdot,i)} \alpha_i \bigg) + c_p \bigg)
		~ \bigg| ~ \alpha_i \in [-1,1], \lambda \in [0,1] \bigg\} , \\
	&= \bigg\{ 0.5 \Big( \big(I_n + e^{A\Delta t_k}\big)c_h + c_p \Big) \\
	&\quad\, + 0.5 \lambda \Big( \big(e^{A\Delta t_k} - I_n \big) c_h + c_p \Big) + 0.5 \sum_{i=1}^{\gens{}_h} G^{(+)}_{h(\cdot,i)} \alpha_i \\
	&\quad\, 
		+ 0.5 \sum_{i=1}^{\gens{}_h} G^{(-)}_{h(\cdot,i)} \alpha_i \lambda ~
		\bigg| ~ \alpha_i, \lambda \in [-1,1] \bigg\} ,
\end{align*}
with $G_h^{(+)} = (e^{A\Delta t_k} + I_n) G_h$ and $G_h^{(-)} = (e^{A\Delta t_k} - I_n) G_h$.
In order to represent this exact linear combination $\S{}$ as a zonotope, we have to substitute the bilinear factors $\alpha_i \lambda$ in the last term by additional linear factors $\omega_i \in [-1,1]$.
By neglecting the dependency between $\alpha$ and $\lambda$, we obtain the outer-approximation $\Sover{}$ of the exact linear combination $\S{}$.
Due to the obvious containment $\S{} \subseteq \Sover{}$, the formula for the Hausdorff distance in \eqref{eq:dH} simplifies to
\begin{equation*} %\label{eq:dH_SZ}
	\dH{\S{},\Sover{}} = \max_{\widehat{s} \in \Sover{}} \min_{ \substack{\vspace{1pt} \\ s \in \S{} } } \norm{\widehat{s} - s}_2 .
\end{equation*}
Exploiting the identical factors before and after conversion, all terms but one cancel out and we obtain
\begin{align}
	&\max_{\widehat{s} \in \Sover{}} \min_{ \substack{\vspace{1pt} \\ s \in \S{}} } \norm{\widehat{s} - s}_2
	\leq \max_{\substack{\omega_i \in [-1,1] \\ \alpha_i \in [-1,1] \\ \lambda \in [-1,1]} } 0.5 \norm{ \sum_{i=1}^{\gens{}_h} G^{(-)}_{h(\cdot,i)} (\omega_i - \alpha_i \lambda) }_2 \nonumber \\
	&\qquad \qquad \leq 0.5 \norm{ G^{(-)}_h }_2
		\max_{\varphi \in [-\vecones{},\vecones{}]} \norm{ (\omega - \alpha \lambda) }_2
		\label{eq:dH_SSover}
\end{align}
with $\alpha = [\alpha_1~\dots~\alpha_{\gens{}_h}]^\top$, $\omega = [\omega_1~\dots~\omega_{\gens{}_h}]^\top$, and $\varphi = [\omega~\alpha~\lambda]^\top$.
According to the Bauer Maximum Principle, we may assume that the maximum is attained at a point which satisfies the constraints
\begin{equation*}
	\forall i \in \{1,...,\gens{}_h\}: \omega^2_i = 1, \alpha^2_i = 1, \quad \text{and} \quad \lambda^2 = 1
\end{equation*}
for the maximization term in \eqref{eq:dH_SSover}.
Consequently, we obtain
\begin{align*}
	&\max_{\varphi \in [-\vecones{},\vecones{}]} \norm{ \omega - \alpha \lambda }^2_2
	= \max_{\varphi \in [-\vecones{},\vecones{}]}
		\omega^\top \omega + \alpha^\top \alpha \,\lambda^2
		- 2 \lambda \,\omega^\top \alpha \\
	&\leq \max_{\varphi \in [-\vecones{},\vecones{}]} \omega^\top \omega
		+ \max_{\varphi \in [-\vecones{},\vecones{}]} \alpha^\top \alpha \,\lambda^2
		+ \max_{\varphi \in [-\vecones{},\vecones{}]} -2 \lambda \,\omega^\top \alpha \\
	&= \gens{}_h + \gens{}_h + 2 \gens{}_h = 4 \gens{}_h,
\end{align*}
which implies $\max_{\varphi \in [-\vecones{},\vecones{}]} \norm{ \omega - \alpha \lambda }_2 \leq 2 \sqrt{\gens{}_h}$.
We insert this result into \eqref{eq:dH_SSover} to obtain the error
\begin{equation*}
	\dH{\S{},\Sover{}}
%	\leq \frac{ \norm{ G^{(-)}_h }_2}{2} \, \max_{\varphi \in [-\vecones{},\vecones{}]}
%		\norm{ \omega - \alpha \lambda }_2
	\leq 0.5 \norm{ G^{(-)}_h }_2 \, 2 \sqrt{\gens{}_h}
	= \sqrt{\gens{}_h} \norm{G^{(-)}_h}_2 ,
\end{equation*}
which concludes the proof.
\end{proof}

To derive the error $\eaccUstep{k}$ for one time step contained in the particular solution $e^{At_k} \overPU{\Delta t_k}$ as used in \propref{prop:eaccU}, we require the following proposition:

\begin{proposition} \label{prop:dH_MS}
For a compact set $\mathcal{S} \subset \R{n}$ and two matrices $M_1, M_2 \in \R{m \times n}$, we have
\begin{equation*} %\label{eq:dH_MS}
	\dH{(M_1 + M_2) \mathcal{S}, M_1 \mathcal{S}} \leq \dH{\matzeros{}, M_2 \mathcal{S}} .
\end{equation*}
\end{proposition}

\begin{proof}
For the left-hand side, we choose $s_1 = s_2$ for both $\min$-operations in the definition \eqref{eq:dH} of the Hausdorff distance:
\begin{align*}
	&\dH{(M_1 + M_2) \mathcal{S}, M_1 \mathcal{S}} \\
	&\quad \overset{\eqref{eq:dH}}{=} \max \Big\{ \max_{s_1 \in \S{}}
			\big( \, \min_{s_2 \in \S{}} \, \norm{(M_1+M_2) s_1 - M_1 s_2}_2 \,
		\big), \\
	&\quad \qquad \qquad \max_{s_2 \in \S{}}
		\big( \, \min_{s_1 \in \S{}} \norm{(M_1+M_2) s_1 - M_1 s_2}_2 \,
		\big) \Big\} \\
	&\quad \leq \max \Big\{
		\max_{s_1 \in \S{}} \norm{(M_1+M_2) s_1 - M_1 s_1}_2, \\
	&\quad \qquad \qquad \max_{s_2 \in \S{}} \norm{(M_1+M_2) s_2 - M_1 s_2}_2
	\Big\} \\
	&\quad = \max_{s \in \S{}} \norm{M_2 s}_2 .
\end{align*}
The right-hand side evaluates trivially to
\begin{equation*}
	\dH{\matzeros{}, M_2 \S{}} = \max_{s \in \S{}} \norm{M_2 s}_2 ,
\end{equation*}
which combined with the result above yields the claim.
\end{proof}

Using \propref{prop:dH_MS}, we obtain the following error bound:

\begin{proposition} \label{prop:errUoneStep}
The Hausdorff distance between the propagated exact particular solution
\begin{equation*}
	e^{A t_k} \PU{\Delta t_k} = \bigg\{ e^{A t_k} \int_0^{\Delta t_k} e^{A(\Delta t_k-\theta)} u(\theta) \, \mathrm{d}\theta ~ \bigg| ~ u(\theta) \in \inputsetzero{} \bigg\}
\end{equation*}
and the outer-approximation $e^{At_k} \overPU{\Delta t_k}$ with $\overPU{\Delta t_k}$ from \eqref{eq:overPU_init} is bounded by \eqref{eq:eaccUstep}.
\end{proposition}
\begin{proof}
We obtain a tight bound for the error by computing the distance between an inner-approximation $\underPU{\Delta t_k}$ and the outer-approximation $\overPU{\Delta t_k}$ in \eqref{eq:overPU_init}.
By considering uncertain but constant inputs we compute an inner-approximation $\underPU{\Delta t_k} \subseteq \PU{\Delta t_k}$ as
\begin{align*}
	\underPU{\Delta t_k} &= \int_{0}^{\Delta t_k} e^{A(\Delta t_k - \theta)} \mathrm{d}\theta \, \inputsetzero{} \\
	&= \bigg( \sum_{i=0}^{\infty} \frac{A^{i} \Delta t_k^{i+1}}{(i+1)!} \bigg) \, \inputsetzero{}
	= \bigg( \Delta t_k \, I_n + \sum_{i=1}^{\infty} \tilde{A}_i \bigg) \, \inputsetzero{}
\end{align*}
where $\tilde{A}_i = \frac{A^i \Delta t_k^{i+1}}{(i+1)!}$.
Applying \propref{prop:dH_MS} with $M_1 = \Delta t_k I_n$ and $M_2 = \sum_{i=1}^{\infty} \tilde{A}_i$, we have
\begin{align*}
	&\dH{\underPU{\Delta t_k},\Delta t_k \, \inputsetzero{}}
	\leq d_H \bigg( \matzeros{}, \bigg( \sum_{i=1}^{\infty} \tilde{A}_i \bigg) \, \inputsetzero{} \bigg) \\
	&\qquad \overset{\eqref{eq:dH<=err}}{\leq}
		\errOp{ \bigg( \sum_{i=1}^{\infty} \tilde{A}_i \bigg) \, \inputsetzero{} } \\
	&\qquad \overset{\eqref{eq:E}}{\leq} \errOp{ \bigg( \sum_{i=1}^{\tayl{}} \tilde{A}_i \bigg) \, \inputsetzero{}
		\oplus \E{k} \Delta t_k \, \inputsetzero{}} .
\end{align*}
From \eqref{eq:overPU_init}, we have the trivial containment $\Delta t_k \, \inputsetzero{} \subset \overPU{\Delta t_k}$ and thus
% note: one might need to prove that dH(A,A \oplus B) \leq err(B)...
%
\begin{align*}
	&\dH{\Delta t_k \, \inputset{},\overPU{\Delta t_k}} \\
	&\quad \leq \errOp{ \bigoplus_{i=1}^{\tayl{}} \bigg( \tilde{A}_i \, \inputsetzero{} \bigg) 
		\oplus \E{k} \Delta t_k \, \inputsetzero{}} .
\end{align*}
We use the triangle inequality to combine the last two results:
\begin{align*}
	&\dH{\PU{\Delta t_k},\overPU{\Delta t_k}}
	\leq \dH{\underPU{\Delta t_k},\overPU{\Delta t_k}} \nonumber \\
	&\quad \leq \dH{\underPU{\Delta t_k},\Delta t_k \, \inputsetzero{}}
		+ \dH{\Delta t_k \, \inputsetzero{},\overPU{\Delta t_k}} \nonumber \\
	&\quad \leq \errOp{ \bigg( \sum_{i=1}^{\tayl{}} \tilde{A}_i \bigg) \inputsetzero{}
		\oplus \E{k} \Delta t_k \, \inputsetzero{} } \\
	&\qquad + \errOp{ \bigoplus_{i=1}^{\tayl{}} \bigg( \tilde{A}_i \, \inputsetzero{} \bigg)
		\oplus \E{k} \Delta t_k \, \inputsetzero{} } .
\end{align*}
Including the mapping by $e^{A t_k}$ then yields the final result.
\end{proof}

%% file: subtex/lemmas.tex
\renewcommand{\thesectiondis}[2]{\Alph{section}:}
\section{Additional Lemmata}
\label{app:lemmas}

For the proof of \thmref{thm:automated}, we require results about the limit behavior of the error terms, which we derive here.
First, we examine the remainder of the exponential matrix:
\begin{lemma} \label{lmm:E_lim}
	The remainder of the exponential matrix $\E{k}$ in \eqref{eq:E} satisfies
	\begin{equation*}
		\lim_{\Delta t_k \to 0} \E{k} = \matzeros{}
		\quad \text{and} \quad
		\E{k} \sim \bigO{ \Delta t_k^{\tayl{k}+1} } .
	\end{equation*}
\end{lemma}
\begin{proof}
For the limit value we obtain from \eqref{eq:E}
\begin{align*}
	&\lim_{\Delta t_k \to 0} E(\Delta t_k, \tayl{k}) = e^{|A|0} - \sum_{i=0}^{\tayl{k}} \frac{1}{i!} \big( |A|0 \big)^i = \matzeros{} , \\
	&\lim_{\Delta t_k \to 0} \E{k} = \lim_{\Delta t_k \to 0} [-E(\Delta t_k,\tayl{k}),E(\Delta t_k,\tayl{k})] = \matzeros{} ,
\end{align*}
and the asymptotic behavior follows from
\begin{align*}
	E(\Delta t_k, \tayl{k})
	= e^{|A|\Delta t_k} - \sum_{i=0}^{\tayl{k}} \frac{\big( |A|\Delta t_k \big)^i}{i!}
	= \sum_{i=\tayl{k}+1}^{\infty} \frac{\big( |A|\Delta t_k \big)^i}{i!} ,
\end{align*}
which yields $\E{k} \sim \bigO{ \Delta t_k^{\tayl{k}+1} }$.
\end{proof}

Next, we consider the error of the particular solution:
\begin{lemma} \label{lmm:eaccUstep_lim}
	The error $\eaccUstep{k}$ in \eqref{eq:eaccUstep} satisfies
	\begin{equation*}
		\lim_{\Delta t_k \to 0} \eaccUstep{k} = 0
		\quad \text{and} \quad
		\eaccUstep{k} \sim \bigO{ \Delta t_k^2 }.
	\end{equation*}
\end{lemma}
\begin{proof}
Using the definition of the error $\eaccUstep{k}$ according to \eqref{eq:eaccUstep}
with $\tilde{A}_i = \frac{A^i \Delta t_k^{i+1}}{(i+1)!}$ and \lmmref{lmm:E_lim}, we have
\begin{alignat*}{2}
	&\lim_{\Delta t_k \to 0} \tilde{A}_i = 0 , \quad		&& \tilde{A}_i \sim \bigO{ \Delta t_k^2 }, \\
	&\lim_{\Delta t_k \to 0} \E{k}\Delta t_k = 0 , \quad	&& \E{k}\Delta t_k \sim \bigO{ \Delta t_k^{\tayl{k}+2} },
\end{alignat*}
where we exploit that the minimal index is $i = 1$.
It is then straightforward to obtain the limit and asymptotic behavior for $\eaccUstep{k}$.
\end{proof}

While we require a faster than linear decrease in $\Delta t_k$ for the accumulating error, a linear decrease is sufficient for the non-accumulating error as it only affects a single time step:
\begin{lemma} \label{lmm:ehom_lim}
	The error $\ehom{k}$ in \eqref{eq:ehom} satisfies
	\begin{equation*}
		\lim_{\Delta t_k \to 0} \ehom{k} = 0 \quad \text{and} \quad \ehom{k} \sim \bigO{ \Delta t_k } .
	\end{equation*}
\end{lemma}
\begin{proof}
We examine the two individual terms of $\ehom{k} = 2 \errOp{\C{}} + \sqrt{\gens{}_h} \, \norm{G_h^{(-)}}_2 $ separately:
\begin{align*}
	&\lim_{\Delta t_k \to 0} \sqrt{\gens{}_h} \, \norm{G_h^{(-)}}_2
	= \lim_{\Delta t_k \to 0} \sqrt{\gens{}_h} \, \norm{(e^{A\Delta t_k} - I_n) G_h}_2 \\
	&\qquad = \sqrt{\gens{}_h} \, \norm{(e^{\matzeros{}} - I_n) G_h}_2
	= \sqrt{\gens{}_h} \, \norm{\matzeros{}}_2 = 0 .
\end{align*}
To analyze the asymptotic behavior, it suffices to look at $(e^{A\Delta t_k} - I_n)G_h$ as the matrix $G_h$ and the factor $\gens{}_h$ do not depend on the time step size:
Since $(e^{A\Delta t_k} - I_n) \sim \bigO{\Delta t_k}$, we consequently obtain
\begin{equation*}
	\sqrt{\gens{}_h} \, \norm{G_h^{(-)}}_2 \sim \bigO{ \Delta t_k } .
\end{equation*}
For the limit behavior of the term $2 \errOp{\C{}}$, we have
\begin{align*}
	&\lim_{\Delta t_k \to 0} \mathcal{I}_i(\Delta t_k) \overset{\eqref{eq:I}}{=} \mathcal{I}_i(0) = \big[\big( i^{\frac{-i}{i-1}} - i^{\frac{-1}{i-1}} \big) 0^i,0\big] = 0 , \\
	&\lim_{\Delta t_k \to 0} \Fx{k} \overset{\eqref{eq:Fx}}{=} \bigoplus_{i=2}^{\tayl{k}} \mathcal{I}_i(0) \, \frac{A^i}{i!} \oplus \Ezero{k} = \matzeros{} , \\
	&\lim_{\Delta t_k \to 0} \Fu{k} \overset{\eqref{eq:Fu}}{=} \bigoplus_{i=2}^{\tayl{k}+1} \mathcal{I}_i(0) \, \frac{A^i}{i!} \oplus \Ezero{k} 0 = \matzeros{} ,
\end{align*}
which entails
\begin{equation*}
	\lim_{\Delta t_k \to 0} 2 \underbrace{( \errOp{ \Fxzero{} \HPtp{t_k} } + \errOp{ \Fuzero{} \uTrans{} } )}_{\auxerr(\C{})} = 0 .
\end{equation*}
Moreover, we have $\mathcal{I}_i(\Delta t_k) \sim \bigO{ \Delta t_k^2 }$ as the minimal index is $i=2$, so that we obtain in combination with \lmmref{lmm:E_lim}
\begin{equation*}
	2 ( \errOp{ \Fxzero{} \HPtp{t_k} } + \errOp{ \Fuzero{} \uTrans{} } ) \sim \bigO{ \Delta t_k^2 } .
\end{equation*}
It is then straightforward to obtain the limit and asymptotic behavior for $\ehom{k}$.
\end{proof}

\begin{lemma} \label{lmm:eUstep_lim}
	The error $\eUstep{k}$ in \eqref{eq:eUstep} satisfies
	\begin{equation*}
		\lim_{\Delta t_k \to 0} \eUstep{k} = 0
		\hspace{4pt} \text{and} \hspace{4pt}
		\eUstep{k} \sim \bigO{ \Delta t_k }.
	\end{equation*}
\end{lemma}
\begin{proof}
For the limit, we insert $\Delta t_k = 0$ into \eqref{eq:overPU_init} to obtain
\begin{align*}
	&\lim_{\Delta t_k \to 0} \eUstep{k} = \errOp{e^{At_k} \overPU{0}} \\
	&\quad = \errOp{e^{At_k} \bigg( \bigoplus_{i=0}^{\tayl{k}} \frac{A^i 0^{i+1}}{(i+1)!} \, \inputsetzero{}
		\oplus \Ezero{k} 0 \, \inputsetzero{} \bigg)} = 0 .
\end{align*}
According to \lmmref{lmm:E_lim}, we have $\E{k} \sim \bigO{\Delta t_k}$, and with the first term of the sum above, we obtain $\eUstep{k} \sim \bigO{\Delta t_k}$.
\end{proof}

%% file: subtex/biographies.tex
\begin{IEEEbiography}[{\includegraphics[width=1in,height=1.25in,clip,keepaspectratio]{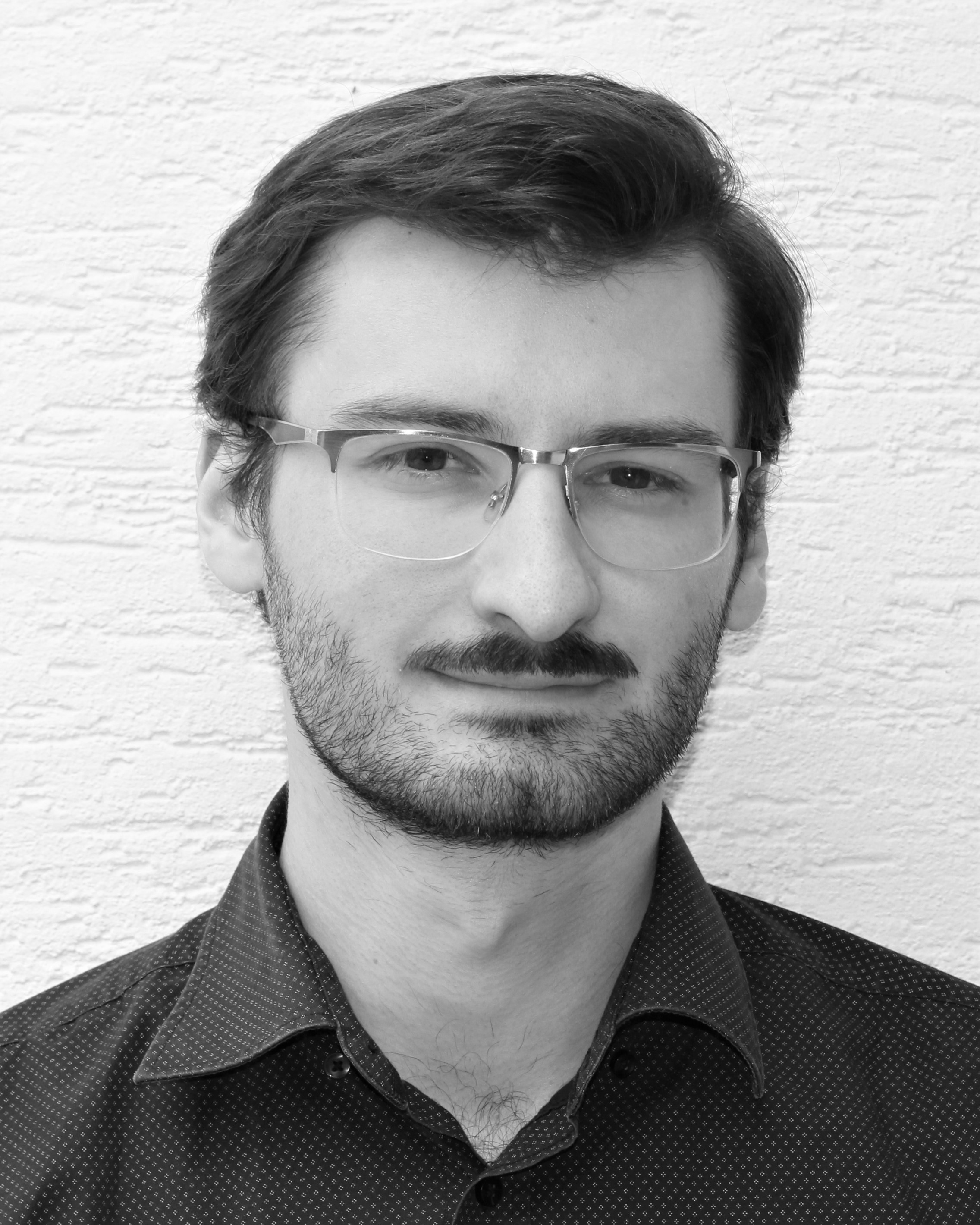}}]{Mark Wetzlinger} received the B.S. degree in Engineering Sciences in 2017 jointly from Universit\"{a}t Salzburg, Austria and Technische Universit\"{a}t M\"{u}nchen, Germany, and the M.S. degree in Robotics, Cognition and Intelligence in 2019 from Technische Universit\"{a}t M\"{u}nchen, Germany. He is currently pursuing the Ph.D. degree in computer science at Technische Universit\"{a}t M\"{u}nchen, Germany. His research interests include formal verification of linear and nonlinear continuous systems, reachability analysis, adaptive parameter tuning, and model order reduction.
\end{IEEEbiography}
\begin{IEEEbiography}[{\includegraphics[width=1in,height=1.25in,clip,keepaspectratio]{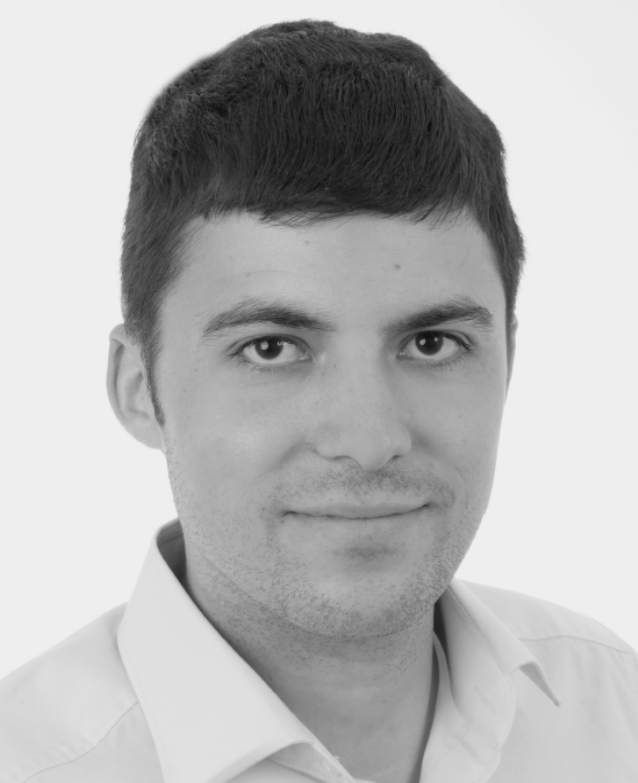}}]{Niklas Kochdumper} received the B.S. degree in Mechanical Engineering in 2015, the M.S. degree in Robotics, Cognition and Intelligence
in 2017, and the Ph.D. degree in computer science in 2022, all from Technische Universit\"{a}t M\"{u}nchen, Germany. He is currently a postdoctoral researcher at Stony Brook University, USA. His research interests include formal verification of continuous and hybrid systems, reachability analysis, computational geometry, controller synthesis, and neural network verification.
\end{IEEEbiography}
\begin{IEEEbiography}[{\includegraphics[width=1in,height=1.25in,clip,keepaspectratio]{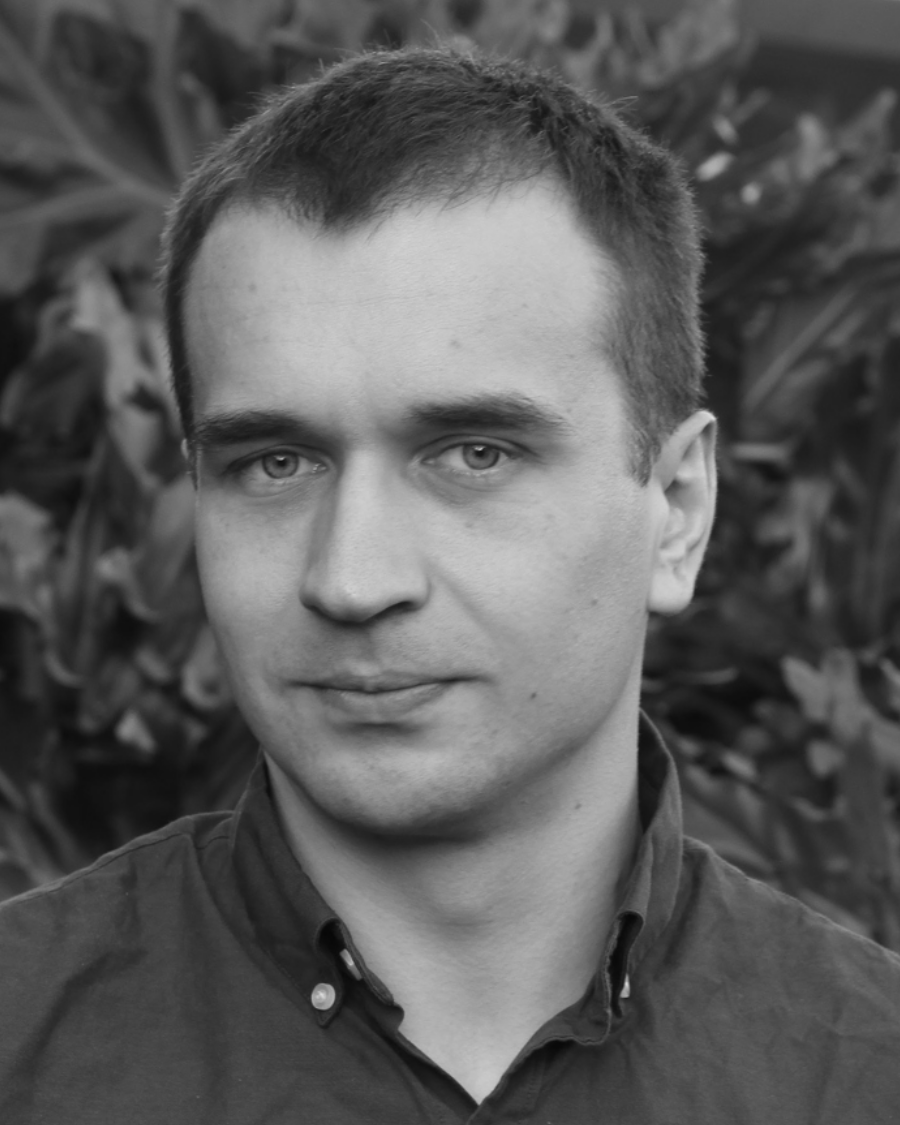}}]{Stanley Bak} is an assistant professor in computer science at Stony Brook University in Stony Brook, NY, USA. He received the B.S. degree in computer science from Rensselaer Polytechnic Institute in 2007, and the M.S. degree and Ph.D. degree both in computer science from the University of Illinois at Urbana-Champaign in 2009 and 2013. His research interests include verification and testing methods for cyber-physical systems and neural networks.
\end{IEEEbiography}
\begin{IEEEbiography}[{\includegraphics[width=1in,height=1.25in,clip,keepaspectratio]{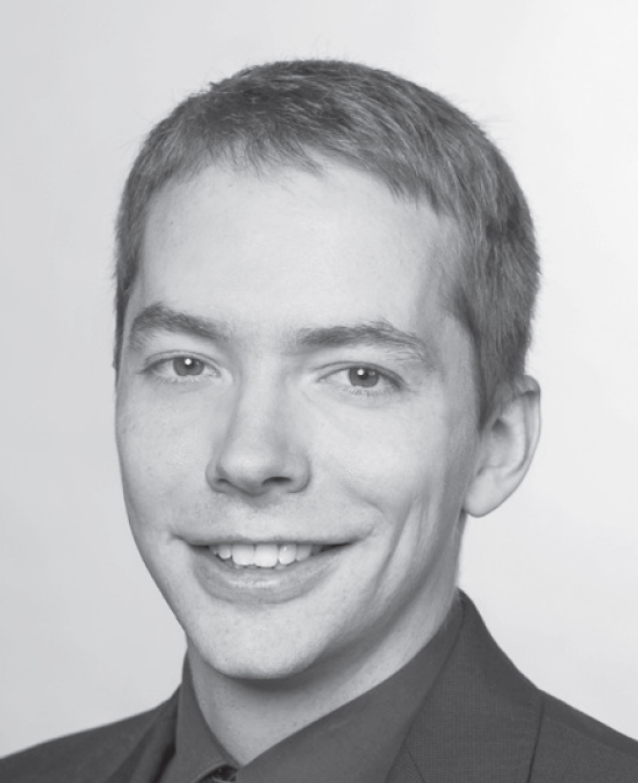}}]{Matthias Althoff} is an associate professor in computer science at Technische Universit\"{a}t M\"{u}nchen, Germany. He received his diploma engineering degree in Mechanical Engineering in 2005, and his Ph.D. degree in Electrical Engineering in 2010, both from Technische Universit\"{a}t M\"{u}nchen, Germany. From 2010 to 2012 he was a postdoctoral researcher at Carnegie Mellon University, Pittsburgh, USA, and from 2012 to 2013 an assistant professor at Technische Universit\"{a}t Ilmenau, Germany. His research interests include formal verification of continuous and hybrid systems, reachability analysis, planning algorithms, nonlinear control, automated vehicles, and power systems.
\end{IEEEbiography}